\documentclass{amsart}
\usepackage{customizations}
\usepackage{cite}
\usepackage{amsmath}
\usepackage{tikz-cd}
\usepackage{hyperref}
\usepackage{accents}

\numberwithin{equation}{section}

\tikzset{
  symbol/.style={
    draw=none,
    every to/.append style={
      edge node={node [sloped, allow upside down, auto=false]{$#1$}}}
  }
}

\begin{document}

\title{Divisor varieties of symmetric products\\\vspace{10pt}\large{John Sheridan}\vspace{-15pt}}
\author{John Sheridan}
\address{Department of Mathematics, Stony Brook University, Stony Brook, New York 11794}
\email{{\tt john.sheridan@stonybrook.edu}}
\date{\today}

\maketitle

\section{Introduction}

The geometry of divisors on algebraic curves has been
studied extensively over the years. The foundational results of this Brill-Noether theory - due to Kempf \cite{kempf}, Kleiman-Laksov \cite{kleiman-laksov},
Griffiths-Harris \cite{griffhar}, Gieseker \cite{gieseker} and Fulton-Lazarsfeld \cite{fullaz} - imply that on a general genus $g$ curve $C$, the varieties $G^r_d(C)$ of degree $d$, dimension $r$ linear series
on $C$ are smooth, irreducible
projective varieties of known dimension depending only on $g$, $d$ and $r$. The theory underlying these results is particularly satisfying
because of its power in connecting the concrete
geometry of curves in projective space with the more abstract notion
of varying a line bundle \textit{continuously} in its cohomology class
- a notion more intrinsic to the curve.\\
\\
In higher dimensions, the story is less well
understood. Lopes-Pardini-Pirola have obtained a Kempf-type existence
result for the Brill-Noether theory of divisors on surfaces in \cite{lpp}. Deformations
of the canonical linear series have been studied by making use of the
generic vanishing theorem of Green-Lazarsfeld \cite{GL1,GL2} and some related
foundational results on the so-called \textit{paracanonical system}
were given by Lopes-Pardini-Pirola in
\cite{lpp2}, extending earlier results of Beauville \cite{be} and
Lazarsfeld-Popa \cite{lazpop}. Our purpose here is to study in detail
one class of higher dimensional examples where one can hope for a
quite detailed picture, namely (the spaces of) divisors on the symmetric product of a curve.\\
\\
Turning to details, let $C$ be a smooth complex projective curve of genus $g$, and denote by $C_k$ the $k^{\text{th}}$ symmetric product of $C$, a smooth projective variety of dimension $k$. We will be interested in two distinguished types of line bundle on $C_k$ arising from a line bundle $L$ on $C$. First, $L$ determines a line bundle $T_L$ by symmetrizing the product $L^{\boxtimes k}$ along the quotient map $C^k \rightarrow C_k$. We find $H^0(C_k,T_L) = S^kH^0(C,L)$. Such an $L$ also gives rise to an ``anti-symmetric'' line bundle $N_L$ on $C_k$ arising as the determinant of the tautological vector bundle $E_{k,L} := L^{[k]}$ on $C_k$. One has $H^0(C_k,N_L) = \wedge^kH^0(C,L)$. This latter bundle in particular has enjoyed much interest in the literature, e.g. in \cite{einlaz} and \cite{einlazyang}. We will write $n(d) := c_1(N_L)$ and $t(d) := c_1(T_L)$ for $L$ of degree $d$.\\
\\
Given a N\'{e}ron-Severi class $\lambda \in \text{NS}(X)$, denote by $\text{Div}^{\lambda}(X)$, $\text{Pic}^{\lambda}(X)$ the spaces of effective divisors and line bundles, respectively, of class $\lambda$. The Abel-Jacobi map $u:\text{Div}^{\lambda}(X) \rightarrow \text{Pic}^{\lambda}(X)$ sends $D \mapsto \mathscr{O}_X(D)$. Recall that $\text{Div}^{\lambda}(X)$ can be realized as $\mathbb{P}\mathcal{F}_{\lambda}$ for an appropriate Picard sheaf $\mathcal{F}_{\lambda}$ on $\text{Pic}^{\lambda}(X)$ (see, e.g. \cite[Ex. 9.4.7]{fga}).\\
\\
Our results describing the structure of the divisor varieties $\text{Div}^{n(d)}(C_k)$ and $\text{Div}^{t(d)}(C_k)$, our chief objects of study in this paper, form the content of Theorems \ref{irreds} and \ref{inters} and Corollary \ref{components}. We briefly overview the picture in this introduction and then indicate the more refined statements in later sections.\\
\\
The first point to make about these divisor varieties is that it is
common for them to contain ``exorbitant'' components when $k \geq 2$ -
a term introduced by Beauville in \cite{be} (which he credits to
Enriques) to refer to components of the divisor variety \textit{other} than one which dominates $\text{Pic}^{\lambda}(C_k)$. We have:

\begin{theorem}\label{1}
Assume $C$ is a Petri-general curve and let $R_d$ and $r_d$ denote the
maximal and minimal dimensions respectively of $|L|$ for a degree $d$
bundle $L$ on $C$. If $\rho(g,d,R_d) \not= 0$ then the varieties
$\text{Div}^{t(d)}(C_k)$ and $\text{Div}^{n(d)}(C_k)$ have $R_d - r_d$
and $(R_d - r_d) - (k-r_d) = R_d - k$ irreducible components,
repsectively, for $k \geq 3$. And each has $\lceil (R_d-r_d)/2 \rceil
+ \epsilon$
irreducible components for $k = 2$, where $\epsilon = (r_d + 1) \text{ \emph{mod} }2$.
\end{theorem}

It is then natural to ask how these components intersect. Starting with a vector space $V$ one can define a kind of rank-locus called a \textbf{subspace-variety}
$$
\text{Sub}_r(\wedge^kV) := \{[\eta]
\in \mathbb{P}\wedge^kV : \eta \in \wedge^kW \text{
  for some }W \in G(r,V)\}.
$$
$\text{Sub}_r(S^kV)$ is defined similarly inside $\mathbb{P}S^kV$. We then have:

\begin{theorem}\label{2}
When $C$ is Petri-general, the
  intersections of the irreducible components of $\text{Div}^{n(d)}(C_k)$ and $\text{Div}^{t(d)}(C_k)$ are fibered along the Abel-Jacobi map $u$ in \textbf{subspace-varieties} $\text{Sub}_r(\wedge^kH^0(L))$ and $\text{Sub}_r(S^kH^0(L))$ respectively, for $L$ in Brill-Noether loci $W^{r+1}_d(C) \subset \text{Pic}^d(C) \cong \text{Pic}^{\lambda}(C_k)$.
\end{theorem}

\noindent A more precise version of this description will be given in Theorem \ref{inters} and will ultimately be a consequence of Theorems \ref{sheaf-id} and \ref{deform}, which identify the Picard sheaves for each algebraic class $\lambda$ and describe the way in which sections deform as a line bundle varies in $\text{Pic}^{\lambda}(X)$.\\
\\
First though, we take a moment to study some examples. To whet the appetite of the reader, we outline here the simplest examples on $C_2$. By way of preparation, note that a pencil $\pi : C \rightarrow \mathbb{P}^1$ in $|L|$ gives rise to a corresponding \textit{trace divisor} $D_{\pi} := \{\xi \in C_2 : \xi\text{ is in a fiber of }\pi\} \in |N_L|$.

\begin{example}\label{eg}[Paracanonical series]
Suppose $C$ is a smooth curve of genus $g$. The Brill-Noether variety
$G^r_{2g-2}(C)$ dominates $\text{Pic}^{2g-2}(C)$ for $g
\geq 3$ and $1 \leq r \leq g - 2$ and, since paracanonical bundles satify Petri's
condition, it is smooth. This means that any $r$-dimensional linear series $V$ of canonical
divisors on $C$ can be deformed \textit{algebraically} in any
direction $v\in T_{K_C}\text{Pic}^{2g-2}(C)$.\\
\\
For general $g \geq 3$, the key observation is that the \textit{parity} of $g$ determines whether or not \textit{all} canonical
divisors $D \in |K_{C_2}|$ on the symmetric square arise (in a sense to be made precise later) from deforming paracanonical linear series on $C$. If $g$ is odd then they do; if $g$ is even then there are
\textit{extra} canonical divisors on $C_2$. So when $g$ is odd, the paracanonical divisor variety $\text{Div}^{\kappa}(C_2)$ of $C_2$ is irreducible. It is a $\mathbb{P}^{{g-1\choose 2}-1}$-bundle away from $K_{C_2}$, over which it has fiber $\mathbb{P}^{{g\choose 2}-1}$. However, when $g$ is even the divisor variety has \textit{two}
components. In that case we have:
$$
\text{Div}^{\kappa}(C_2) \cong \Sigma \cup |K_{C_2}|
$$
where $\Sigma$ is the unique component which dominates
$\text{Pic}^{\kappa}(C_2)$. Its intersection with the canonical series is:
$$
\Sigma\cap |K_{C_2}| = \text{Sec}_s\mathbb{G}(1,|K_C|)
$$
Here $s = (g-2)/2$, $\mathbb{G}(r,\mathbb{P})$ denotes the Grassmannian of $r$-planes in a projective space $\mathbb{P}$, and $\text{Sec}_sX$ denotes the variety of
$s$-secant-$(s-1)$-planes of $X$ in projective space.\\
\\
The following diagram gives a sense, in the even genus case, of how $\text{Div}^{\kappa}(C_2)$
maps (down) to $\text{Pic}^{\kappa}(C_2)$ via the Abel-Jacobi map:
\hspace{0pt}
\begin{center}
\includegraphics[scale=0.75]{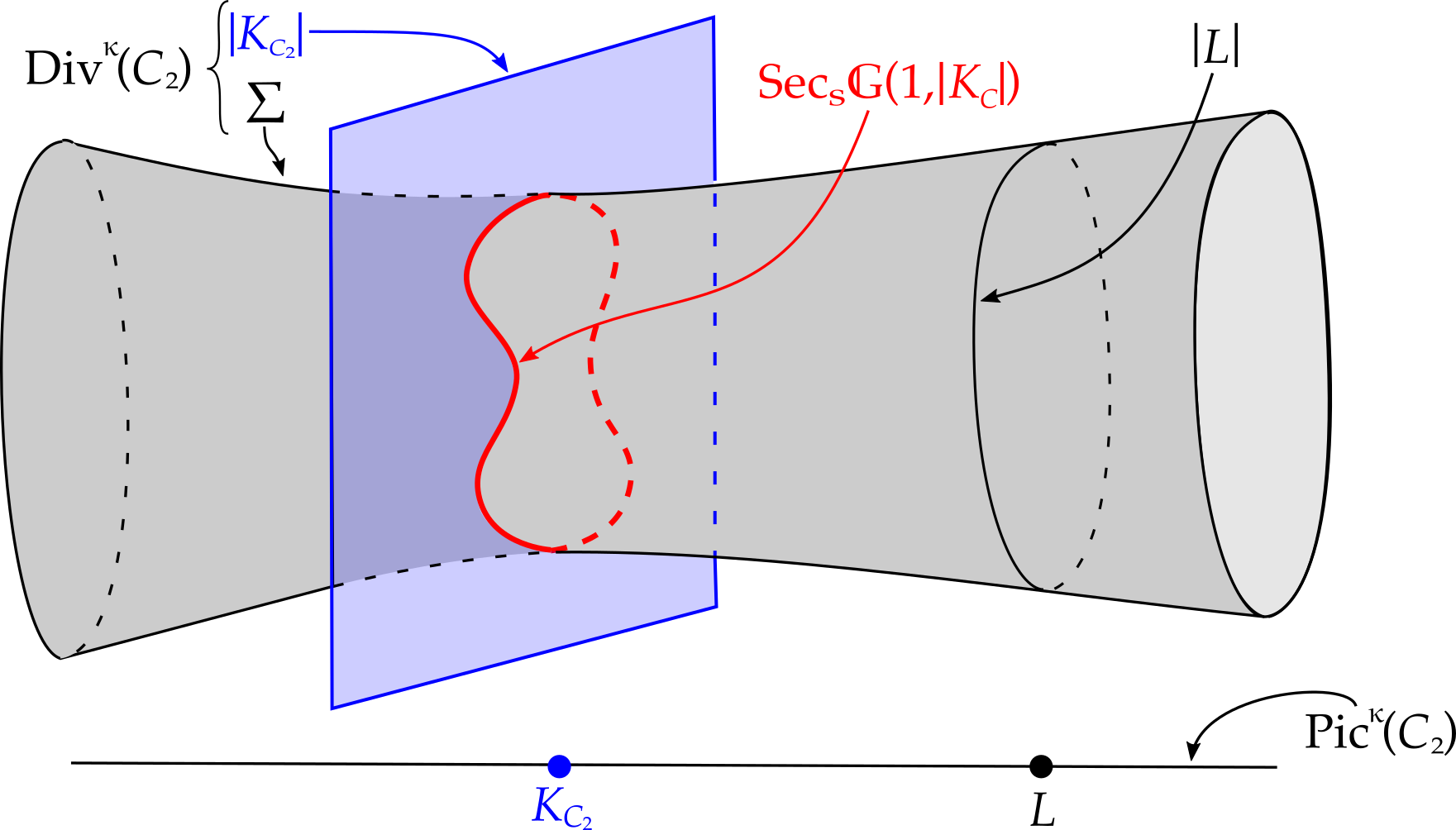}
\end{center}
In particular it is worth considering $g = 4$, in which case $s = 1$ so that the deformable canonical divisors form precisely the Grassmannian $\mathbb{G}(1,3)$ and the irreducible component $\Sigma$ is actually isomorphic to the Brill-Noether variety $G^1_6(C)$ of degree 6 pencils on $C$. So the ``extra'' canonical divisors on $C_2$ in this case are precisely those which are not trace divisors of canonical pencils on $C$ and thus cannot deform to the trace divisors of paracanonical pencils which form the remainder of $\Sigma$ away from $K_{C_2}$.
\end{example}

This example shows, on the one hand, that $K_{C_2}$ is
\textit{exorbitant} (Beauville's term from \cite{be}, mentioned above)
whenever $C$ has even genus, which here means
that the canonical linear series $|K_{C_2}|$ does not lie entirely in
the irreducible component $\Sigma$ of the paracanonical system \textit{dominating}
$\text{Pic}^{\kappa}(C_2)$ (often called the \textit{main
  paracanonical system} or
$\text{Div}^{\kappa}(C_2)_{\text{main}}$). On the other hand, this
example identifies precisely the canonical divisors of $C_2$ which deform, with a deformation of the bundle $K_{C_2}$, to paracanonical ones. The first observation was
predicted by \cite[Thm 1.3]{lpp2}, which extended results in
\cite{be}. The
second observation is a paradigm for the more general results for
divisors on symmetric products that we present in this paper.

\begin{example}[Plane curves]
Let $C$ be a smooth plane curve of degree $d \geq 5$ with $L =
\mathscr{O}_C(1)$. By a result of Marc Coppens \cite[Thm 3.2.1]{copp1}
there are no \textit{free}, complete $g^1_d$'s on $C$. Given this
fact, it can be seen that there are two kinds of $g^1_d$ on $C$:
\begin{enumerate}
\item Given $p \in \mathbb{P}^2$ the pencil of lines through $p$ cuts
  out a $g^1_d$ on $C$. Of course, this $g^1_d$ is a sub-linear system
  of the (unique) $g^2_d$.
\item Given $p\in C$ the pencil of lines in $\mathbb{P}^2$ through $p$
  now determines a $g^1_d$ for which $p$ itself as a basepoint - for a second point $q \in C$
  different from $p$, one can modify the $g^1_d$ by adding $q - p$ to
  every divisor. The result is a \textit{complete} $g^1_d$ for the
  line bundle $L_{pq} := L\otimes \mathscr{O}_C(q-p) \in W^1_d(C)$.
\end{enumerate}
Briefly, in this setting all divisors of class $\lambda = c_1(N_L)$ on $C_2$ are trace divisors of pencils on $C$. They are each of one of the two types above. Among those of type (1) are the divisors which \textit{deform} with a deformation of $L$ (and thus of $N_L$) - those come precisely from choosing $p \in \mathbb{P}^2$ to lie on the curve $C$. We note that the difference map $C^2
\rightarrow \text{Pic}^d(C)$ given by $(p,q) \mapsto L_{pq}$ is
birational onto $W^1_d(C)$ with fiber over $L$ equal to the diagonal $\Delta
\subset C^2$. Hence $G^1_d(C) \cong |L| \cup C^2$ and since all divisors on $C_2$ here are trace divisors, $\text{Div}^{\lambda}(C_2)$
and $G^1_d(C)$ coincide. So we have a decomposition into irreducible components:
$$
\text{Div}^{\lambda}(C_2) \cong \mathbb{P}^2 \cup C^2
$$
with intersection
$$
\mathbb{P}^2 \cap C^2 = \Delta \cong C
$$
Again, we can use a diagram to give a sense for how $\text{Div}^{\lambda}(C_2)$
maps (down) to $\text{Pic}^{\lambda}(C_2)$ via the Abel-Jacobi map:
\vspace{0pt}
\hspace{0pt}
\begin{center}
\includegraphics[scale=0.7]{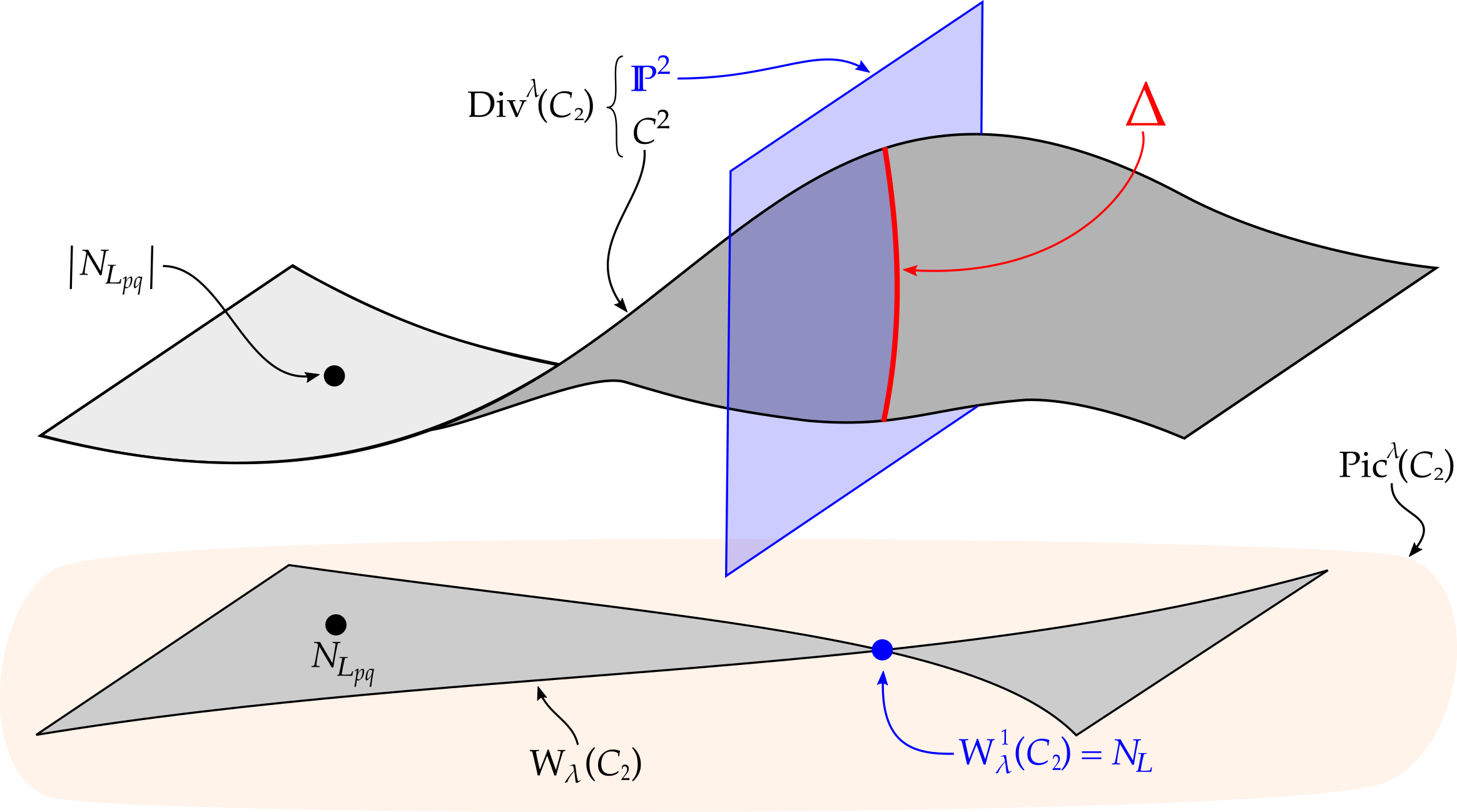}
\end{center}

\end{example}

\vspace{0pt}
This example, and the $g = 4$ case of the previous one, show the trace divisors of pencils on $C$ governing much of the divisor behavior on $C_2$ and give the first hint at the role that will be played by (deforming) higher dimensional linear series on $C$ more generally. Though it is worth pointing out that in the second
example, the
paucity of \textit{free} degree $d$ linear systems on $C$ (due essentially to the
speciality of $C$ in moduli) gives a greater role in this governance to \textit{basepointed}
pencils than is typical.\\
\\
\textbf{Organization of the paper}: In section \ref{prelims} we outline the necessary results describing
how the cohomology of line bundles $N_L$ varies as $L$ varies on the
curve.\\
\indent In section \ref{symprod} we first identify Picard sheaves of the symmetric product $C_k$ and then prove a deformation result for them. We then proceed to prove the main theorems.\\
\indent The final section consists of a variety of examples intended to
illustrate both the general picture of the main theorems as well as
some more specific related ideas.\\
\\
\textbf{Acknowledgements}: I would like to thank my advisor, Robert Lazarsfeld, for suggesting this
very interesting problem and for the many
conversations that helped build my intuition for it. Thanks also to
Frederik Benirschke, Nathan Chen, Fran\c{c}ois Greer, Tim Ryan and
Jason Starr for
valuable discussions and to the Stony Brook math community at large
for a very engaging environment.

\section{Preliminaries}\label{prelims}

\noindent Let $X$ be a smooth, projective variety over $\mathbb{C}$ and $\lambda \in \text{NS}(X)$ a Ner\'{o}n-Severi class. $\text{Pic}^{\lambda}(X)$ will then denote the space of line bundles with first Chern class $\lambda$ and we will define on it a corresponding \textit{Picard sheaf} $\mathcal{F}_{\lambda}$ which, morally, plays the role of a coherent sheaf whose fiber over $L \in \text{Pic}^{\lambda}(X)$ is naturally $H^0(X,L)^{\vee}$. Specifically, we ask that $\mathcal{F}_{\lambda}$ satisfy the property that for any quasi-coherent sheaf $\mathcal{N}$ on $\text{Pic}^{\lambda}(X)$ there is a natural isomorphism of sheaves
\begin{equation}\label{eqn-picard}
q : \text{\underline{Hom}}(\mathcal{F}_{\lambda},\mathcal{N}) \rightarrow \nu_*(\mathscr{L}_{\lambda}\otimes \nu^*\mathcal{N})
\end{equation}
for $\nu : X\times \text{Pic}^{\lambda}(X) \rightarrow \text{Pic}^{\lambda}(X)$ and $\mathscr{L}_{\lambda}$ a Poincar\'{e} line bundle (i.e. a universal line bundle) on $X\times \text{Pic}^{\lambda}(X)$. $\mathcal{F}_{\lambda}$ is only unique up to twisting by a line bundle. This implies that $\text{Div}^{\lambda}(X) \cong \mathbb{P}\mathcal{F}_{\lambda}$ as schemes over $\text{Pic}^{\lambda}(X)$, for $\text{Div}^{\lambda}(X)$ the variety of effective divisors of class $\lambda$. See \cite[9.3.10, p. 260; Ans. 9.4.7, p. 305]{fga} for more details.\\
\\
Our description of divisor varieties of symmetric products will hence follow once we establish two main results about the Picard sheaves. Suppose $\mathcal{F}_d$ is one of the classical Picard sheaves on $\text{Pic}^d(C)$ associated to $C$ and note that $\text{Pic}^d(C) \cong \text{Pic}^{n(d)}(C_k) \cong \text{Pic}^{t(d)}(C_k)$ (see section \S\ref{pic-comps}). First, we will show that
$$
\mathcal{F}_{n(d)} := \wedge^k\mathcal{F}_d \quad \text{and} \quad \mathcal{F}_{t(d)} := S^k\mathcal{F}_d
$$
are Picard sheaves for $C_k$ on $\text{Pic}^{n(d)}(C_k)$ and $\text{Pic}^{t(d)}(C_k)$ respectively. Then we will use this identification to prove a deformation result for sections of $N_L$ (resp. $T_L$) as $L$ varies in $\text{Pic}^{\lambda}(X)$ - this will yield our description of component intersections in $\text{Div}^{\lambda}(X)$.\\
\\
To identify these Picard sheaves, the idea is to globalize the isomorphisms $H^0(C_k, N_L)^{\vee} \cong \wedge^kH^0(C,L)^{\vee}$ and $H^0(C_k,T_L)^{\vee} \cong S^kH^0(C,L)^{\vee}$. The main subtlety warranting caution in this situation is the inevitable \textit{jumping} of $h^0(C,L)$ as $L$ varies.

\subsection{Group actions on coherent sheaves}\label{gp-acts}
\noindent Much of the following material is known, but we include a brief review of what we need for the benefit of the reader.\\
\\
Let $X$ and $Y$ be normal varieties over $\mathbb{C}$ and suppose $X$ admits an algebraic action of a finite group $G$. Let $\pi : X \rightarrow Y$ be a proper $G$-invariant morphism (for our purposes later, this morphism will be the quotient by $G$). If a coherent $\mathscr{O}_X$-module $\mathcal{F}$ admits an action of
$G$ commuting with that on the base (a $G$-equivariant structure), we
define its \textit{symmetrization} or \textit{equivariant pushforward} $\pi^G_*\mathcal{F} := (\pi_*\mathcal{F})^G$ on $Y$ to be the
sheaf of $G$-invariants whose sections over $U \subset Y$ are the
$G$-invariant sections of $\mathcal{F}$ on
$\pi^{-1}(U)$. If $\mathcal{F}$ is locally free, $\pi^G_*\mathcal{F}$
will also be locally free of rank equal to that of $\mathcal{F}$.

\begin{proposition}[Equivariant rank drop]\label{an-lemma}
If $u : E \rightarrow F$ is a $G$-equivariant map of locally free sheaves on $X$ that drops rank exactly along a $G$-invariant divisor $D$, then there is a $G$-equivariant isomorphism:
$$
\xymatrix{\text{det}(u) : \text{det}(E) \ar[r]^{\hspace{-10pt}\cong} & \text{det}(F) \otimes \mathscr{O}_X(-D)}
$$
\end{proposition}
\begin{proof}
The isomorphism alone follows by \cite[Lemma 5.1]{an}. That it is $G$-equivariant is clear from the setup.
\end{proof}

\begin{proposition}[Cohomology and invariants]\label{g-eq}
  Let $X$, $Y$, $\pi$, $\mathcal{F}$ and $G$ be as above and suppose $S$ is another normal variety over $\mathbb{C}$ which fits into the following diagram (with $\pi$ and $\tau$ both $G$-invariant, $\tau$ and $\overline{\tau}$ both flat and projective):
  \begin{center}
    \begin{tikzcd}
      X \arrow[d,,"\pi"] \arrow[dr,,"\tau"]\\
      Y \arrow[r,,"\overline{\tau}"] & S
    \end{tikzcd}
  \end{center}
  Then we can calculate higher direct images of $\pi^G_*\mathcal{F}$ along $\overline{\tau}$ by taking invariants of the corresponding higher direct images along $\tau$ upstairs:
  $$
  R^i\overline{\tau}_*(\pi^G_*\mathcal{F}) \cong (R^i\tau_*\mathcal{F})^G
  $$
\end{proposition}
\begin{proof}
  Since $G$ is finite and the $G$-modules in which $\pi_*\mathcal{F}$ takes its values are over $\mathbb{C}$-algebras, the invariants functor $(\_)^G$ is exact in our situation (a consequence of Maschke's theorem on complete reducibility of $G$-representations). Therefore as a trivial special case of Grothendieck's spectral sequence (see \cite[Theorem 2.4.1]{toh}) we have that $R^i\tau^G_*\mathcal{F} = (R^i\tau_*\mathcal{F})^G$ for all $i$ (because recall that $\tau^G_* = (\_)^G\circ \tau_*$). Similarly $R^i\pi^G_*\mathcal{F} = (R^i\pi_*\mathcal{F})^G$. Since $\pi$ is finite, $R^i\pi_*\mathcal{F} = 0$ for $i > 0$. Now we note that $\tau^G_* = \overline{\tau}_*\circ \pi^G_*$. We can apply the Grothendieck spectral sequence here too to conclude $R^p\overline{\tau}_*(R^q\pi^G_*\mathcal{F})$ abuts to $R^{p+q}\tau^G_*\mathcal{F} = (R^{p+q}\tau_*\mathcal{F})^G$, but since the higher direct images of $\pi$ vanish, this abutment immediately reduces to the desired isomorphism.
\end{proof}

\begin{remark}
The same result could be achieved in the above proposition with weaker hypotheses on $G$, on the spaces $X$, $Y$ and $S$ and for different fields. However, we will only work with the symmetric group over $\mathbb{C}$.
\end{remark}

\subsection{K\"{u}nneth Formula}\label{kunneth-sec}
Suppose we have the following Cartesian diagram of schemes:
$$
\xymatrix{
X\times_S Y \ar[r] \ar[d] \ar[dr]^{\tau} & Y\ar[d]^{g}\\
X \ar[r]^{f} & S
}
$$
where $X$, $Y$ and $S$ are smooth varieties over $\mathbb{C}$ and where $f$ and $g$ are flat of relative dimensions $m$ and $n$ respectively.

\begin{proposition}[Top degree K\"{u}nneth formula]\label{kunneth}
For $X$, $Y$, $f$, $g$, $\tau$ as above, suppose that $\mathcal{F}$ and $\mathcal{G}$ are locally free sheaves on $X$ and $Y$ respectively. Then:
$$
\xymatrix{
(R^mf_*\mathcal{F})\otimes (R^ng_*\mathcal{G}) \ar[r]^{\hspace{10pt}\cong} & R^{n+m}\tau_*(\mathcal{F}\boxtimes\mathcal{G})
}
$$
\end{proposition}
\begin{proof}
By \cite[6.7.6]{ega} the local freeness of $\mathcal{F}$ and $\mathcal{G}$ means that the \textit{K\"{u}nneth spectral sequence} (see \cite[6.7.3(a)]{ega}) computes $R^{n+m}\tau_*(\mathcal{F}\boxtimes \mathcal{G})$. This spectral sequence lies in the fourth quadrant and has $(R^mf_*\mathcal{F})\otimes (R^ng_*\mathcal{G})$ in the bottom-left corner (position $(p,q) = (0,-m-n)$). The result follows from convergence of this corner term.
\end{proof}

Letting $Y = X^{\times (k-1)}$ (the $(k-1)$-fold product of $X$), $g = f^{\times (k-1)}$ and $\mathcal{G} = \mathcal{F}^{\boxtimes (k-1)}$ in the above proposition we get, by induction, an isomorphism $\phi: (R^mf_*\mathcal{F})^{\otimes k} \rightarrow R^{km}\tau_*(\mathcal{F}^{\boxtimes k})$. We will now consider two natural actions of the symmetric group here, one including a twist by the sign homomorphism and the other not. To cover both cases simultaneously, we will let the symbol $\mu$ stand for either $0$ or $1$ and consider the symmetric group $\mathfrak{S}_k$ to act by $(-1)^{\mu}$ times the natural permutation action on $\mathcal{F}^{\boxtimes k}$. This induces an action on $R^{km}\tau_*(\mathcal{F}^{\boxtimes k})$. Letting $\mathfrak{S}_k$ then act on $(R^mf_*\mathcal{F})^{\otimes k}$ by $(-1)^{m+\mu}$ (to account for skew-symmetry of odd-degree cohomology) times the natural permutation action makes $\phi$ an \textit{equivariant} map. Hence we have:

\begin{proposition}[Equivariant K\"{u}nneth]\label{eq-kunneth}
Letting $\mathfrak{S}_k$ act as indicated above, $\phi$ restricts to an isomorphism between the invariant subsheaves:
\begin{align*}
\xymatrix{
\wedge^kR^mf_*\mathcal{F} \ar[r]^{\hspace{-10pt}\cong} & R^{km}\tau_*(\mathcal{F}^{\boxtimes k})^{\mathfrak{S}_k}} & \hspace{20pt}\text{for $m + \mu$ odd}\\
\xymatrix{S^kR^mf_*\mathcal{F}\ar[r]^{\hspace{-10pt}\cong} & R^{km}\tau_*(\mathcal{F}^{\boxtimes k})^{\mathfrak{S}_k}} & \hspace{20pt}\text{for $m + \mu$ even}
\end{align*}
\end{proposition}

\subsection{Brill-Noether loci and Petri-general curves}\label{bn-petri}

\noindent Let $C$ be a smooth projective curve over $\mathbb{C}$.\\
\\
\indent In the remainder of this paper we will often refer to the classical \textit{Brill-Noether loci}
$$
W^r_d(C) := \{L \in \text{Pic}^d(C) : h^0(C,L) \geq r + 1\}
$$
and their natural desingularizations,
$$
G^r_d(C) := \{V \subset H^0(C,L) : L \in \text{Pic}^d(C),\quad\text{dim }V = r + 1\}
$$
which we will refer to here as the \textit{Brill-Noether varieties}. Both are defined as schemes in \cite[Ch. 4, \S 3]{acgh1}.\\
\\
\indent To state our results below in the appropriate generality, we will also need the following:
\begin{definition}\label{petri-map}
  We say $C$ is a \textit{Petri curve} if for all line bundles $L$ on $C$, the Petri-map:
  $$
\mu_0 : H^0(L)\otimes H^0(K_C-L) \rightarrow H^0(K_C)
$$
(given by multiplication of sections) is \textit{injective}.
\end{definition}
\noindent It is a celebrated theorem of Gieseker (\cite[Theorem 1.1]{gieseker}) that a general curve (in the sense of moduli) is a Petri curve. Without further reference, we will say \textit{Petri-general curve} to mean any curve in an open subset of the set (in moduli) of Petri curves.\\
\\
\indent Now, in order to count components of our divisor varieties below, we would like to know the size of the set
\begin{align*}
  \mathcal{R}_d &:= \{r \in \mathbb{Z}_{\geq 0} : h^0(L) = r + 1\text{
                  for some }L \in \text{Pic}^d(C)\}\\
  \\
  & = \{r \in \mathbb{Z}_{\geq 0} : W^r_d\setminus W^{r+1}_d \not= \varnothing\}
\end{align*}
Recall that on a Petri-general curve $C$ the existence and dimension
theorems for $g^r_d$'s (\cite[p. 206]{acgh1} and \cite[p. 214]{acgh1})
imply that $W^r_d$ is non-empty of dimension $\text{min}\{g,\rho(g,r,d)\}$ if and only
if the Brill-Noether number $\rho(g,d,r)$ (see \cite[p. 159]{acgh1})
is non-negative. So, by these facts and the above definition, one can
determine that the loci $W^r_d$
are distinct for different values of $r$ in a certain range: for $r_d := \text{min }\mathcal{R}_d$ and $R_d := \text{max }\mathcal{R}_d$ their count is
\begin{equation}\label{wrd-count}
|\mathcal{R}_d| = R_d - r_d + 1
\end{equation}
with
 \begin{align*}
   R_d & = \left\lfloor \frac{d - g - 1 + \sqrt{(d - g - 1)^2 +
         4d}}{2} \right\rfloor\\
   \\
   r_d & =  \left\{\begin{array}{cl}d-g& \text{for }d>g\\\\0& \text{for }d \leq g\end{array}\right.
\end{align*}
(here $R_d$ is the floor of the largest root of the Brill-Noether number
$\rho(g,d,r)$ thought of as a polynomial in $r$).\\
\\
In particular here, we note that $|\mathcal{R}_d|$ is approximately \textit{linear} in $d$ (for fixed genus).

\begin{remark}
  Although not
immediately obvious from this formula, it is not hard to
see that for $d \geq 2g - 1$ we always have $|\mathcal{R}_d| = 1$, as
expected from Riemann-Roch which in that case implies $\mathcal{R}_d
= \{d - g\}$.
\end{remark}

\subsection{Subspace varieties and their desingularizations}\label{lin-alg}

\noindent Let $V$ be a vector space over $\mathbb{C}$ and recall the \textit{coproduct} map $\wedge^{\bullet}V \rightarrow \wedge^{\bullet}V\otimes \wedge^{\bullet}V$ from multilinear algebra (see e.g. \cite[p. 3]{weyman}). After projecting, this yields:
\begin{center}
\begin{tikzcd}[row sep = tiny]
\Delta : & [-25pt] \wedge^kV \arrow[r,,] & \wedge^{k-1}V\otimes V\\
& u_1\wedge \cdots \wedge u_k \arrow[r,mapsto] & \sum (-1)^{\text{sgn}(\sigma)}u_{\sigma(1)}\wedge \cdots \wedge u_{\sigma(k-1)}\otimes u_{\sigma(k)}
\end{tikzcd}
\end{center}
where the sum is taken over all permutations $\sigma \in \mathfrak{S}_k$ such that $\sigma(1) < \cdots < \sigma(k-1)$.\\
\\
\indent We now introduce the notion of the \textit{enclosing space} $\text{Enc}(\eta)$ of a $k$-vector $\eta \in \wedge^kV$:

\begin{definition}\label{enc}
  Let $V$ be a vector space, $\eta \in \wedge^kV$ and $\Phi \in S^kV$. We define the \textit{enclosing spaces} $\text{Enc}(\eta)$ and $\text{Enc}(\Phi)$ to be the smallest subspaces $U$ and $U'$ respectively such that $\eta \in \wedge^kU \subset \wedge^kV$, and similarly for $\Phi$ and $U'$. We will denote the dimensions of these enclosing spaces by $\text{enc}(\eta)$ and $\text{enc}(\Phi)$.\\
  \\
  \indent Equivalently (see e.g. \cite[p. 210-211]{griffhar2}), we have the algebraic definition that $\text{Enc}(\eta)$ and $\text{Enc}(\Phi)$ are the images, respectively, of the following contraction maps:
\begin{align*}
\langle \_\hspace{1pt},\eta\rangle & : \wedge^{k-1}V^* \rightarrow V\\
\langle \_\hspace{1pt},\Phi\rangle & : S^{k-1}V^* \rightarrow V
\end{align*}
Explicitly, the contraction map $\langle \_ , \eta\rangle$ is the composition of the maps
\begin{center}
\begin{tikzcd}
\wedge^{k-1}V^* \arrow{r}{\text{$-\otimes \Delta(\eta)$}} & [2em] \wedge^{k-1}V^*\otimes \wedge^{k-1}V\otimes V \arrow{r}{\text{$t\otimes -$}} & [2em] V
\end{tikzcd}
\end{center}
for $t$ the trace map $\wedge^{k-1}V^*\otimes \wedge^{k-1}V \rightarrow \mathbb{C}$ (respectively for the symmetric powers).
\end{definition}

Given the above definition of enclosing spaces, we are naturally led
to consider the following parameter spaces (note that unless otherwise
indicated, the symbol $\mathbb{P}$ will denote taking projective
\textit{quotients}; projective \textit{subspaces} will be denoted by $\mathbb{P}_{\text{sub}}$):

\begin{definition}\label{subspace-varieties}
Let $V$ be a vector space of dimension $n$ and let $k \leq e \leq \text{dim }V$ some positive integers. We define the (skew-)symmetric \textbf{subspace varieties}:
\begin{align*}
\text{Sub}_e(\wedge^kV) & := \{[\eta]\in \mathbb{P}(\wedge^kV)^{\vee} : \text{enc}(\eta) \leq e\}\\
\text{Sub}_e(S^kV) & := \{[\Phi] \in \mathbb{P}(S^kV)^{\vee} : \text{enc}(\Phi) \leq e\}
\end{align*}
\end{definition}

Note that when $e = k$ in the skew-symmetric case we get the Grassmannian $G(k,V) = \text{Sub}_k(\wedge^kV)$ and when $e = 1$ in the symmetric case we get the Veronese variety $\nu_k(\mathbb{P}V) = \text{Sub}_1(S^kV)$.
\begin{remark}
For $k = 2$ the subspace varieties are \textit{secant varieties} $\text{Sec}_s\hspace{1pt}G(2,V) = \text{Sub}_{2s}(\wedge^2V)$ and $\text{Sec}_s\hspace{1pt}\nu_2(\mathbb{P}V) = \text{Sub}_{2s}(S^2V)$ (for $\nu_2$ the quadratic Veronese mapping). This is because the enclosing dimension of a (skew-) symmetric 2-tensor coincides with the rank of the corresponding (skew-) symmetric matrix.
\end{remark}

\begin{remark}\label{coincidences}
  A priori it is possible that two subspace varieties will coincide for different choices of enclosing dimension $e$:
  \begin{itemize}
  \item For $k = 2$, by the previous remark, we have $\text{Sub}_{2s}(\wedge^2V) = \text{Sub}_{2s+1}(\wedge^2V)$ for $1 \leq s \leq \lfloor n/2\rfloor$.
  \item For $k \geq 3$, these coincidences happen rarely: $\text{Sub}_k(\wedge^kV) = \text{Sub}_{k+1}(\wedge^kV)$ but otherwise $\text{Sub}_e(\wedge^kV) \subsetneq \text{Sub}_{e+1}(\wedge^kV)$ for all $k + 1 \leq e \leq n$.
    \end{itemize}
\end{remark}

Typically the subspace variety $\text{Sub}_e$ will be singular along $\text{Sub}_{e-1}$, but it admits a useful desingularization which is a particular case of a more general construction we briefly outline here: note that the incidence correspondence
  $$
\Psi := \{([\eta],W) \in \mathbb{P}(\wedge^kV)^{\vee} \times G(e,V) : \eta \in \wedge^kW\}
$$
maps surjectively to $\text{Sub}_e(\wedge^kV) \subset \mathbb{P}(\wedge^kV)^{\vee}$. In fact, the fiber over $[\eta]$ is exactly
$$
\{W \in G(e,V) : \text{Enc}(\eta) \subset W\}
$$
hence (when we are not in the situations of Remark \ref{coincidences}) the map is an isomorphism over the open subset $\{[\eta]\in\text{Sub}_e(\wedge^kV) : \text{enc}(\eta) = e\}$ and is thus birational. We note that in fact $\Psi = \mathbb{P}(\wedge^k\mathcal{S})^{\vee}$ for $\mathcal{S}$ the tautological sub-bundle on $G(e,V)$, hence it is a desingularization of $\text{Sub}_e(\wedge^kV)$. A desingularization of $\text{Sub}_e(S^kV)$ can be constructed analogously. These desingularizations immediately imply:

\begin{lemma}\label{dimension}
  For $k \geq 3$, the subspace varieties are irreducible and have dimensions
  \begin{align*}
    \text{\emph{dim}}\left(\text{\emph{Sub}}_e(\wedge^kV)\right) & = e(n - e) + {e\choose k} - 1\\
    \text{\emph{dim}}\left(\text{\emph{Sub}}_e(S^kV)\right) & = e(n-e) + {e + k - 1\choose k} - 1
  \end{align*}
for $n = \text{dim }V$ and $k \leq e \leq n$ except $e = k+1$.
\end{lemma}
The analogous irreducibility statement follows by the same argument in the case $k = 2$, but the dimension is calculated differently since the subspace varieties in that case are defective secant varieties, so the incidence correspondence above is no longer a desingularization.

\begin{lemma}\label{sec-dim}
  The subspace varieties $\text{Sub}_{2s}(\wedge^2V) = \text{Sec}_s\hspace{1pt}G(2,V)$ and $\text{Sub}_{2s}(S^2V) = \text{Sec}_s\hspace{1pt}\nu_2(\mathbb{P}V)$ are irreducible of dimensions
  \begin{align*}
    \text{\emph{dim}}\left(\text{\emph{Sec}}_s\hspace{1pt}G(2,V)\right) & = \text{\emph{min}}\left\{{n\choose 2} - 1, \hspace{5pt}2(n-2)s + s - 1\right\} - 2s(s-1)\\
    \text{\emph{dim}}\left(\text{\emph{Sec}}_s\hspace{1pt}\nu_2(\mathbb{P}V)\right) & = \text{\emph{min}}\left\{{n + 1\choose 2} - 1, {s + 1\choose 2} + s(n - s) - 1\right\}
  \end{align*}
\end{lemma}
\begin{proof}
See \cite{cgg} and \cite[p. 125]{land} respectively for the dimension calculations of the defective secant varieties.
\end{proof}

\begin{remark}\label{kempf-res}
One can apply \cite[Prop. 1 and Thm. 3]{kempf2} to the
map $\Psi \rightarrow \text{Sub}_e(\wedge^kV)$ to conclude that the
subspace variety $\text{Sub}_e(\wedge^kV)$ is normal and
Cohen-Macaulay, \textit{and} - in the case $k\geq 3$ when this map is
birational - it also has rational singularities.
\end{remark}

\section{Symmetric Products}\label{symprod}
\noindent Let $C$ be a smooth projective curve over $\mathbb{C}$.

\subsection{The bundles of interest on $C_k$}
Let $C^k$ and $C_k$ be the $k^{\text{th}}$ direct and symmetric products of $C$, respectively. For $L$ a line bundle on $C$ we can form on $C^k$ the associated rank-$k$
vector bundle $L^{\boxplus k} := \bigoplus p_i^*L$ and its determinant
$L^{\boxtimes k} := \bigotimes p_i^*L$ for $p_i$ the projections. Now
consider the following commutative diagram:

\begin{equation}\label{diag1}
  \begin{tikzcd}
    C\times C_k & & C^k \arrow[dd,"\pi"]\arrow[dl,"\pi_i"]\\
    \mathcal{D} \arrow[u,symbol=\subseteq]\arrow[drr,"q"']
    \arrow[d,"p"'] \arrow[r,symbol=\cong] & C\times C_{k-1} \\
    C & & C_k
  \end{tikzcd}
\end{equation}
where $\pi$ denotes the quotient map, $p$ and $q$ the restrictions of
the projection maps on $C\times C_k$, and $\pi_i$ denotes the map $(x_1,\ldots,x_k) \mapsto
(x_i\hspace{1pt},\hspace{1pt}x_1 + \cdots + \widehat{x_i} + \cdots +
x_k)$. Here $\mathcal{D}$ denotes the \textit{universal degree $k$ divisor on
$C$} which is easily seen to be isomorphic to $C\times C_{k-1}$. We
make the following definition:

\begin{definition}\label{taut-defs}
On $C_k$ we first define $E_L := q_*p^*L$ and from this we define the
determinant bundle $N_L := \text{det}(E_L)$. We
also define $T_L$ to be the symmetrization (see section \ref{gp-acts})
of $L^{\boxtimes k}$ (for $\mathfrak{S}_k$ acting by the natural
permutation action) i.e. $T_L := \pi_*^{\mathfrak{S}_k}L^{\boxtimes k}$.
\end{definition}

\noindent In the literature, $E_L$ is often also denoted $L^{[k]}$ and
referred to as the \textit{tautological rank-$k$ bundle associated to
  $L$} on $\text{Hilb}^k(C) \cong C_k$. It is well-known both that $K_{C_k} = N_{K_C}$ and that $H^0(C_k,N_L)$
and $H^0(C_k,T_L)$ are isomorphic to $\wedge^kH^0(C,L)$ (see \cite[Ch. 5]{an}) and $S^kH^0(C,L)$ respectively.

\begin{remark}
The line bundles $T_L$ and $N_L$ are the bundles of chief interest to
us on $C_k$ and they satisfy the relation $N_L \cong T_L(-\Delta/2)$ for $\Delta \subset C_k$ the
(big) diagonal.\footnote{Here $\Delta \subset C_k$ is the image under
  $\pi$ of the big diagonal in $C^k$. In particular, it is the branch locus of $\pi$. For $f :
X\rightarrow Y$ finite between smooth $X$ and $Y$, the natural
pullback morphism $\mathscr{O}_Y \rightarrow f_*\mathscr{O}_X$ is
\textit{split} by $1/d$ times the trace map (for $d =
\text{deg}(f)$). The dual $E_{f}$ of the remaining summand of
$f_*\mathscr{O}_X$ (known as the \textit{Tschirnhausen bundle} of $f$) has determinant $L_{f} := \text{det}(E_{f})$ which
squares to the normal bundle of the branch locus of $f$ though need
not be effective itself. So
$L_{\pi}$ is what we really mean by $\Delta/2$.} We put $N_L$ and
$T_L$ (respectively, families thereof) on a more equal footing with respect to symmetrization via the following:
\end{remark}

\begin{proposition}
Let $S$ be any normal variety. For $\pi : C^k\times S \rightarrow C_k\times S$ the quotient map by the symmetric group $\mathfrak{S}_k$ and $\mathscr{L}$ any line bundle on $C\times S$, we have
$$
\pi^{\mathfrak{S}_k}_*\mathscr{L}^{\boxtimes k} \cong \left\{\begin{array}{cl} T_{\mathscr{L}} & \mu = 0\\ N_{\mathscr{L}} & \mu = 1\end{array}\right.
$$
for $\mathfrak{S}_k$ acting on $\mathscr{L}^{\boxtimes k}$ by $(-1)^{\mu}$ times the natural permutation action.
\end{proposition}
\begin{proof}
The claim for $\mu = 0$ is true simply by Definition \ref{taut-defs}. For $\mu
= 1$ we proceed as follows: the natural evaluation map
$\pi^*E_{\mathscr{L}} \rightarrow \mathscr{L}^{\boxplus k}$ on $C^k$
(obtained, e.g. for $S$ a point, by pulling back the evaluation map $q^*E_L \rightarrow p^*L$
along $\pi_i$ in Figure \ref{diag1} and summing over $i$) is $\mathfrak{S}_k$-equivariant, hence by Proposition \ref{an-lemma} we have $\pi^*N_{\mathscr{L}} \cong \mathscr{L}^{\boxtimes k}(-\Delta)$ equivariantly. Since  $\pi^{\mathfrak{S}_k}_*\pi^*$ is the identity for coherent sheaves on $C_k$ (see e.g. \cite[Lemma 2.1]{krug}), we apply $\pi^{\mathfrak{S}_k}_*$ to this isomorphism to get $N_{\mathscr{L}} \cong \pi^{\mathfrak{S}_k}_*(\mathscr{L}^{\boxtimes k}(-\Delta))$. Finally, $\mathscr{L}^{\boxtimes k}(-\Delta)$ is an equivariant subsheaf of $\mathscr{L}^{\boxtimes k}$ for $-1$ times the permutation structure on the latter - the quotient is equipped with the $-\text{id}$ representation, which has no invariants, hence $\pi^{\mathfrak{S}_k}_*(\mathscr{L}^{\boxtimes k}(-\Delta)) \cong \pi^{\mathfrak{S}_k}_*\mathscr{L}^{\boxtimes k}$. This completes the proof.
\end{proof}

\subsection{Picard components}\label{pic-comps} As alluded to near the beginning of section \ref{prelims}, the Picard components $\text{Pic}^{n(d)}(C_k)$ and $\text{Pic}^{t(d)}(C_k)$ are isomorphic to $\text{Pic}^d(C)$ in a natural way - briefly, the isomorphisms and their inverses are:
\begin{center}
  \begin{tikzcd}[row sep = 0.1pt]
    \text{Pic}^d(C) \arrow{r}{\text{$\cong$}} & \text{Pic}^{t(d)}(C_k) \arrow{r}{\text{$\cong$}} & \text{Pic}^{n(d)}(C_k)\\
    L \arrow[r,mapsto] & T_L \arrow[r,mapsto] & T_L(-\Delta/2)\\
    N_L(\Delta/2)\vert_{i(C)} & N_L(\Delta/2) \arrow[l,mapsto] & N_L\arrow[l,mapsto]
  \end{tikzcd} 
\end{center}
where for $D \in C_{k-1}$ any degree $k-1$ divisor on $C$ we define $i : C \rightarrow C_k$ by $p \mapsto p + D$ (the isomorphisms are independent of the choice of $D$). So from now on, without further comment, we will identify these various corresponding Picard components and consider the Brill-Noether \textit{loci} $W^r_d = W^r_d(C)$ as subschemes of $\text{Pic}^{n(d)}(C_k)$ and of $\text{Pic}^{t(d)}(C_k)$ where convenient.

\subsection{Identifying Picard sheaves} We are now ready to state and prove our first main theorem:
\begin{theorem}\label{sheaf-id}
Let $C$ be a smooth projective curve and for $d \in \mathbb{Z}$ denote by $\mathcal{F}_d$ a corresponding Picard sheaf on $\text{Pic}^d(C)$. Then
\begin{align*}
\mathcal{F}_{n(d)} & := \wedge^k\mathcal{F}_d\\
\mathcal{F}_{t(d)} & := S^k\mathcal{F}_d
\end{align*}
are Picard sheaves associated to $C_k$ for the Picard components $\text{Pic}^{n(d)}(C_k)$ and $\text{Pic}^{t(d)}(C_k)$ respectively.
\end{theorem}
\noindent Before proving the theorem, we need the following lemma:
\begin{lemma}\label{highest-push-pic}
For $X$ a smooth projective variety of dimension $n$, $\mathscr{L}_{\lambda}$ a universal line bundle on $X\times \text{Pic}^{\lambda}(X)$ and $p : X\times \text{Pic}^{\lambda}(X) \rightarrow X$ and $\nu : X\times \text{Pic}^{\lambda}(X) \rightarrow \text{Pic}^{\lambda}(X)$ the projections, we have that the highest direct image $R^n\nu_*(p^*K_X\otimes \mathscr{L}_{\lambda}^{\vee})$ satisfies equation \ref{eqn-picard} and is thus a Picard sheaf.
\end{lemma}
\begin{proof}
This follows from an application of relative duality in \cite[Theorem 21]{reldual}.
\end{proof}

\begin{proof}[Proof of theorem \ref{sheaf-id}]
Let $S := \text{Pic}^d(C) \cong \text{Pic}^{n(d)}(C_k) \cong \text{Pic}^{t(d)}(C_k)$. Consider the commutative diagram
$$
\xymatrix{
C^k\times S\ar[d]^{\pi}\ar[dr]^{\tau} \ar[r]^{p_i} & C\times S\ar[d]^{\nu}\\
C_k\times S \ar[r]^{\overline{\tau}} & S
}
$$
for $\pi$ the quotient map and the remaining maps just the natural projections (any of the $k$ choices for the top map is valid).\\
\\
By lemma \ref{highest-push-pic}, the sheaves $R^k\overline{\tau}_*(\widetilde{p}_1^*K_{C_k}\otimes\mathscr{N}^{\vee})$ and $R^k\overline{\tau}_*(\widetilde{p}_1^*K_{C_k}\otimes \mathscr{T}^{\vee})$ are Picard sheaves on $\text{Pic}^{t(d)}(C_k)$ and $\text{Pic}^{n(d)}(C_k)$, respectively, for \textit{any} Poincar\'{e} bundles $\mathscr{N}$ and $\mathscr{T}$ for $n(d)$ and $t(d)$ respectively. Here $\widetilde{p}_1 : C_k\times S \rightarrow C_k$ denotes the projection. For $\mathscr{L}$ a degree $d$ Poincar\'{e} bundle for $C$, we note that $\mathscr{N}$ and $\mathscr{T}$ can be chosen to be $N_{\mathscr{L}}$ and $T_{\mathscr{L}}$ respectively. Moreover, $K_{C_k} = N_{K_C}$ and for any $L$ on $C$, we have $N_L = T_L(-\Delta/2)$. Letting $\mathscr{K} := p_1^*K_C \otimes \mathscr{L}^{\vee}$ we thus have that $R^k\overline{\tau}_*(T_{\mathscr{K}})$ and $R^k\overline{\tau}_*(N_{\mathscr{K}})$ are Picard sheaves for $n(d)$ and $t(d)$ respectively (here $p_1 : C\times S \rightarrow C$ is the projection).\\
\\
Now since $\pi$ is finite (hence has vanishing higher direct images), we note that by Proposition \ref{g-eq} we have that $R^k\overline{\tau}_*(T_{\mathscr{K}})$ and $R^k\overline{\tau}_*(N_{\mathscr{K}})$ are isomorphic to $R^k\tau_*(\mathscr{K}^{\boxtimes k})^{\mathfrak{S}_k}$, for $\mathfrak{S}_k$ acting by $(-1)^{\mu}$ times the natural permutation action on $\mathscr{K}^{\boxtimes k}$, for $\mu = 0$ and $1$ respectively (again using $\mu$ to deal with both $\mathfrak{S}_k$-actions simultaneously as in section \ref{kunneth-sec}). So, since $\mathcal{F}_d = R^1\nu_*\mathscr{K}$, we have
\begin{align*}
\xymatrix{
\wedge^k\mathcal{F}_d \ar[r]^{\hspace{-20pt}\cong} & R^{k}\tau_*(\mathscr{K}^{\boxtimes k})^{\mathfrak{S}_k}} & \hspace{20pt}\text{for $\mu = 0$}\\
\xymatrix{S^k\mathcal{F}_d\ar[r]^{\hspace{-20pt}\cong} & R^{k}\tau_*(\mathscr{K}^{\boxtimes k})^{\mathfrak{S}_k}} & \hspace{20pt}\text{for $\mu = 1$}
\end{align*}
by Proposition \ref{eq-kunneth}. Since the targets of these maps are the relevant Picard sheaves (by the previous paragraph), this finishes the proof.
\end{proof}

\subsection{Deformations over $W^r_d$}\label{sym-def}

The result of this section will be the key to identifying the intersections of the irreducible components of $\text{Div}^{\lambda}(C_k)$. In what follows, $W^r_d$ and $G^r_d$ will denote the varieties of interest from Brill-Noether theory associated to a fixed curve $C$, as introduced in section \ref{bn-petri}.

\begin{definition}\label{deform-div-bundle}
Let $X$ be a smooth, projective variety with a line bundle $L$ and
suppose $\mathscr{L}$ is a line bundle on $X\times S$, viewed as a
family of line bundles $\{\mathscr{L}_t\}$ on $X$ deforming $L = \mathscr{L}_0$, parametrized by $t
\in S$ (an integral scheme) and inducing a non-constant map $\gamma : S \rightarrow \text{Pic}(X)$. We will say an effective divisor $D \in |L|$ \textit{deforms with $L$ (over $S$)} if $D$ extends to a divisor $\mathscr{D}\in |\mathscr{L}|$ which is flat over $S$.
\end{definition}

Recall the definition of $\text{Enc}(\eta)$ for $\eta \in \wedge^kV$ and $V$ some vector space (Definition \ref{enc}). Suppose $L$ is a line bundle on $C$ and $D \in |N_L|$ is a divisor on $C_k$. In what follows, since $H^0(C_k,N_L)\cong \wedge^kH^0(C,L)$, we will let $\text{Enc}(D)$ denote $\text{Enc}(\eta)$ for $\eta \in H^0(C_k,N_L)$ any section such that $\text{Zeroes}(\eta) = D$.\\
\\
Recall also that by Remark \ref{coincidences}, the maximum enclosing dimension of $\eta \in \wedge^kV$ is $\text{dim }V$ \textit{unless} either $\text{dim }V = k+1$, or $k = 2$ and $\text{dim }V$ is odd - in both cases the maximum enclosing dimension is $\text{dim }V - 1$. Let $e(k,n)$ denote the maximum enclosing dimension of elements of $\wedge^k\mathbb{C}^n$. So we have:
$$
e(k,n) = \left\{\begin{array}{cl}n - 1 & k = n-1;\quad k = 2\text{ and }n\text{ odd}\\n & \text{otherwise}\end{array}\right.
$$

\begin{theorem}\label{deform}
  Let $C$ be a Petri-general curve of genus $g$, and $d,r$ integers such that $\rho(g,r,d) > 0$. Let $L \in W^r_d$ be a line bundle on $C$ and $D \in |N_L|$ a divisor on $C_k$. Suppose $\{\mathscr{L}_t\}_{t \in S}$ is a one-parameter deformation of $L = \mathscr{L}_0$ inducing a non-constant map $\gamma : S \rightarrow \text{\emph{Pic}}^d(C)$ such that $\gamma(S\setminus 0) \subset W^r_d\setminus W^{r+1}_d$. Then $D$ deforms with (the corresponding deformation of) $N_L$ if and only if
  $$
\text{\emph{enc}}(D) \leq e(k,r+1)
  $$
\end{theorem}
\begin{proof}
Suppose first that $D$ deforms with $N_L$ and therefore extends to a flat family $\{D_t\}_{t \in S}$ for some irreducible curve $S$ with $D_0 = D$ at $0 \in S$. In what follows it suffices to suppose $S$ is smooth since base-changing along the normalization of a singular $S$ will not change the divisors $D_t$.\\
\\
Let $\mathscr{D}$ denote the total space of the family $\{D_t\}$ inside $C_k\times S$ and let $\eta \in H^0(C_k\times S,\mathscr{O}(\mathscr{D}))$ be a defining section of $\mathscr{D}$ whose restriction to $N_{\mathscr{L}_t}$ is \textit{also} a defining section of $D_t$ (such a section is always available by \cite[Lemma 9.3.4]{fga}). The following diagram illuminates the setup:
\begin{center}
\begin{tikzcd}
\mathscr{D} \arrow[r,symbol=\subseteq] & C_k\times S \arrow[dr,,"\tau"] & & C\times S\arrow[dl,,"\nu"']\\
& & S \arrow[rr,,"\gamma"'] & & \text{Pic}^d(C)
\end{tikzcd}
\end{center}
The family $\{\mathscr{L}_t\}$ yields a line bundle $\mathscr{L}$ on $C\times S$ and induces the map $\gamma : S \rightarrow \text{Pic}^d(C)$, and $N_{\mathscr{L}}$ is the line bundle $\mathscr{O}(\mathscr{D})$ on $C_k\times S$. Let $\eta_t := \eta\vert_{C_k\times t} \in H^0(C_k,N_{\mathscr{L}_t})$ for $t \in S$. Just as $\eta_t$ induces the map $\langle \eta_t,\_\rangle : \wedge^{k-1}H^0(C,\mathscr{L}_t)^{\vee} \rightarrow H^0(C,\mathscr{L}_t)$ (see section \ref{lin-alg}) whose image is $\text{Enc}(D_t)$, the section $\eta \in H^0(S,\tau_*N_{\mathscr{L}})$ induces a \textit{relative} version of this map:
\begin{center}
\begin{tikzcd}
\wedge^{k-1}\mathcal{F} \arrow[r,"\overline{\eta}"] & \nu_*\mathscr{L}
\end{tikzcd}
\end{center}
where $\mathcal{F}$ denotes the pullback of the Picard sheaf along $\gamma$.\\
\\
If $\overline{\eta}(t)$ denotes the induced map $\wedge^{k-1}\mathcal{F}\otimes \mathbb{C}(t) \rightarrow (\nu_*\mathscr{L})\otimes \mathbb{C}(t)$ on fibers, then $\overline{\eta}(t) = \langle \eta_t,\_\rangle$ for all $t\not= 0$ (since $h^0(\mathscr{L}_t)$ is \textit{constant} on that locus).\\
\\
Since $\nu_*\mathscr{L}$ is torsion-free on a smooth curve $S$, it is locally free. Similarly, since subsheaves of torsion-free sheaves are also torsion-free, $\text{im}(\overline{\eta})$ is also locally free. Hence $\text{rank}(\overline{\eta}(t))$ is constant.\\
\\
Denoting by $\phi_t : (\nu_*\mathscr{L})\otimes \mathbb{C}(t) \rightarrow H^0(\mathscr{L}_t)$ the base-change map for $\nu$ at $t$, one can check that the following diagram commutes for \textit{all} $t \in S$:
\begin{center}
\begin{tikzcd}
\wedge^{k-1}H^0(\mathscr{L}_t)^{\vee} = \wedge^{k-1}\mathcal{F}\otimes \mathbb{C}(t) \arrow[dr,"{\langle\eta_t,\_\rangle}"] \arrow[rr,"\overline{\eta}(t)"] & & (\nu_*\mathscr{L})\otimes \mathbb{C}(t)\arrow[dl,"\phi_t"']\\
& H^0(\mathscr{L}_t)
\end{tikzcd}
\end{center}
and hence $\text{rank }\langle \eta_t, \_\rangle \leq \text{rank }\overline{\eta}(t)$, with equality if $\phi_t$ is injective.\\
\\
Since $\text{enc}(D_t) = \text{rank }\langle \eta_t,\_\rangle \leq h^0(\mathscr{L}_t)$ and $\text{rank }\overline{\eta}(t)$ is constant $\leq e(k,r+1)$, we get the desired inequality:
$$
\text{enc}(D_0) \leq \text{rank}(\text{im}(\overline{\eta})) \leq e(k,r+1)
$$
\indent For the converse statement, we suppose instead that $D$ is such that $\text{enc}(D)
  \leq e(k,r+1)$. We need to produce a deformation of $D$
  over an irreducible curve $S$ admitting a non-constant classifying map
  $\gamma : S \rightarrow W^r_d \subset \text{Pic}^d(C)$.\\
\\
We note that by the dimension condition, $\text{Enc}(D) \subset
  V \subset H^0(C,L)$ for $\mathbb{P}_{\text{sub}}V \subset |L|$ a $g^r_d$ on
  $C$. If $C$ is Petri-general, then $G^r_d(C)$ is smooth of dimension $\rho(g,r,d)$ (the Brill-Noether number) and $\text{dim }W^r_d = \text{dim }G^r_d$ (\cite[p. 214]{acgh1}) so since $\rho(g,r,d) > 0$ we can pick
  $S \subset G^r_d$ a smooth curve centered at $[\mathbb{P}_{\text{sub}}V] \in G^r_d$ and
  mapping birationally to a curve in $W^r_d$.\\
\\
By the universal property for $G^r_d$ we get a corresponding
  family of $g^r_d$'s on $C$ over $S$. This consists of the data of a
  line bundle $\mathscr{L}$ on $C\times S$ and a subsheaf $E\subset
  \nu_*\mathscr{L}$ locally free of rank $r + 1$ with \textit{injective} maps
  $\iota_t: E_t\rightarrow H^0(C,\mathscr{L}_t)$ for all $t \in S$ (see \cite[p. 184]{acgh1}). In
  particular, note that $E_0 = V$.\\
\\
We choose $\eta_D \in H^0(C_k,N_L)$ such that $D =
  \text{Zeroes}(\eta_D)$ and note that, by definition of
  $\text{Enc}(D)$, $\eta_D \in \wedge^k\text{Enc}(D) \subset \wedge^kV
  = (\wedge^kE)_0$.\\
\\
After possibly replacing $S$ with an open subset containing $0$,
  we can assume $E$ is trivial with a section $\eta \in
  H^0(\wedge^kE)$ which is non-vanishing on $S$ and such that $\iota_0(\eta(0)) =
 \eta_D$.\\
\\
Now we note that since $E \subset \nu_*\mathscr{L}$ we have an
  inclusion $\wedge^kE \subset \wedge^k\nu_*\mathscr{L}$. We also have
 a natural map\footnote{obtained in an analogous way to those in Proposition \ref{eq-kunneth} though this particular map need not be an isomorphism} $\wedge^k\nu_*\mathscr{L} \rightarrow
  \tau_*N_{\mathscr{L}}$ which
  means we can think of $\eta$ (above) as lying in $H^0(S,
  \tau_*N_{\mathscr{L}}) \cong H^0(C_k\times S,N_{\mathscr{L}})$. As
  an element of the latter section space, we
  can therefore use it to define $\mathscr{D} := \text{Zeroes}(\eta)$
  which will be a family of divisors $\{D_t\}_{t \in S}$ where $D_t =
  \text{Zeroes}(\iota_t(\eta(t)))$ (in particular $D_0 =
  D$), which is \textit{flat} since the maps $\iota_t$ are
  \textit{injective} and therefore ensure that $\iota_t(\eta(t)) \not=
  0$ for all $t \in S$ (see \cite[Lemma 9.3.4]{fga}).\\
  \\
  Hence we have produced the necessary family and conclude that
  $D$ indeed deforms with $N_L$.
\end{proof}

\begin{remark}The same deformation result follows analogously in the symmetric case for $N_L$ replaced by $T_L$ and the number $e(k,n)$ replaced by the maximum enclosing dimension $e'(k,n)$ of elements of $S^k\mathbb{C}^n$ which has values
$$
e'(k,n) = \left\{\begin{array}{cl}n - 1 & k = 2\text{ and }n\text{ odd}\\n & \text{otherwise}\end{array}\right.
$$
\end{remark}

\subsection{Divisor varieties on $C_k$}
Identification of the Picard sheaves for $\text{Pic}^{n(d)}(C_k)$ and $\text{Pic}^{t(d)}(C_k)$ together with the conclusion about how the divisors deform enables us to now give rather complete descriptions of $\text{Div}^{n(d)}(C_k)$ and $\text{Div}^{t(d)}(C_k)$ including: the number and dimensions of the irreducible components and the nature of their pairwise intersections.

\begin{definition}
For positive integers $d, e, k$, N\'{e}ron-Severi class $\lambda = n(d)$ or $t(d)$, and $u : \text{Div}^{\lambda}(C_k) \rightarrow \text{Pic}^{\lambda}(C_k) \cong \text{Pic}^d(C)$ the Abel-Jacobi map, we define the following closed subvarieties of $\text{Div}^{\lambda}(C_k)$:
$$
(C_k)^e_{\lambda} := \overline{\{D \in \text{Div}^{\lambda}(C_k) : \text{enc}(D) \leq e \text{ and }u(D) \in W^{e-1}_d\setminus W^e_d\}}
$$
where the closure is taken in $\text{Div}^{\lambda}(C_k)$.
\end{definition}
We will see in a moment that, for appropriate choices of $e$, these subvarieties form the irreducible components of $\text{Div}^{\lambda}(C_k)$ and we can count their number by determining how many distinct possibilities there are for the enclosing dimension as $\mathscr{O}_{C_k}(D)$ varies in $\text{Pic}^{\lambda}(C_k)$.\\
\\
Heuristically, consider a divisor varying continuously in $\text{Div}^{\lambda}(C_k)$ (we will think of $\lambda = n(d)$ for the moment, but $\lambda = t(d)$ is analogous). If $D$ moves continuously to $D'$, this induces $L$ to move continuously to $L'$ (for $L$ and $L'$ the line bundles on $C$ such that $D \in |N_L|$ and $D'\in |N_{L'}|$). If the line bundle moves from $W^r_d$ into $W^s_d$ (for $s > r$) then $\text{enc}(D) \leq r + 1$ and by Theorem \ref{deform} $\text{enc}(D') \leq r + 1$ too. Conversely, the same theorem essentially says that if $D' \in |N_{L'}|$ and $L'$ moves from $W^s_d$ out to $L \in W^r_d$, then $D'$ can move with $L'$ to some $D \in |N_L|$ \textit{as long as} $\text{enc}(D') \leq r + 1$.\\
\\
Reasoning in this way for every element of a fixed linear system $|N_L|$, we will conclude that as $L$ moves from $W^r_d$ to $L' \in W^s_d$, the elements of the linear system $|N_L|$ \textit{all} move into $\text{Sub}_m(\wedge^kH^0(L'))$ for $m = e(k,h^0(L))$, and vice versa.\\
\\
The idea is therefore that, as $L$ moves in $\text{Pic}^d(C)$ in such a way that $h^0(L)$ increases from $r_d + 1$ to $R_d + 1$ (recall from section \ref{bn-petri} that these are the minimum and maximum section counts in $\text{Pic}^d(C)$) and the linear systems $|L|$ trace out $\text{Div}^{n(d)}(C_k)$, this divisor variety will pick up new components exactly when the number $e(k,h^0(L))$ \textit{jumps}. We record these jumps by defining $\mathcal{E}(k,\lambda)$ to be the set of these jump values in the symmetric ($\lambda = t(d)$) and skew-symmetric ($\lambda = n(d)$) cases:
\begin{align*}
  \mathcal{E}(k,n(d)) & := \{e \in \mathbb{N} : e = e(k,h^0(L))\text{ for some }L \in \text{Pic}^d(C)\}\\
  \mathcal{E}(k,t(d)) & := \{e \in \mathbb{N} : e = e'(k,h^0(L))\text{ for some }L \in \text{Pic}^d(C)\}
\end{align*}
By Remark \ref{coincidences} we can determine these sets precisely (note that the values are \textit{even} when $k = 2$):
\begin{align*}
  \mathcal{E}(2,t(d)) & := \{e \in 2\mathbb{N} : r_d + 1 \leq e\leq R_d + 1\} & \mathcal{E}(k,t(d)) & := \{e \in \mathbb{N} : r_d + 1\leq e \leq  R_d + 1\}\\
\mathcal{E}(2,n(d)) & := \{e \in \mathcal{E}(2,t(d)) : e \geq 2\} &  \mathcal{E}(k,n(d)) & := \{e \in \mathcal{E}(k,t(d)) : e \geq k,\hspace{2pt} e \not= k + 1\}
\end{align*}
\noindent With these definitions in place, we are ready to state the following results:

\begin{theorem}[Irreducible components]\label{irreds}
  Let $C$ be a smooth projective Petri-general curve of genus $g$ over $\mathbb{C}$ and let $d, k \in \mathbb{Z}_{>0}$. Assume $\rho(g,R_d,d) \not=0$ (see Corollary \ref{components}). Then for $\lambda = n(d)$ or $t(d)$ and $\mathcal{E}(k,\lambda)$ as above:
    $$
    \text{\emph{Div}}^{\lambda}(C_k) = \bigcup_{e\in \mathcal{E}(k,\lambda)}(C_k)^{e}_{\lambda}
    $$
    are decompositions into irreducible components.
\end{theorem}

\begin{proof}
That the divisor variety $\text{Div}^{\lambda}(C_k)$ is precisely the stated union of subvarieties is almost immediate - we need only confirm that any divisor $D$ such that $\text{enc}(D) \leq e$ but $u(D) \in W^e_d$ is still in this union, for $e \in \mathcal{E}(k,\lambda)$. This follows from Theorem \ref{deform} which says that since $\text{enc}(D) \leq e$ there is a one-parameter family of divisors $\{D_t\}_{t \in S}$ such that $D_0 = D$, $\text{enc}(D_t) = e$ for $t\not= 0$ and such that $\gamma(S\setminus 0) \subset W^{e-1}_d\setminus W^e_d$ for $\gamma$ the classifying map for the family $\{\mathscr{O}_{C_k}(D_t)\}$ of line bundles. Hence $D \in (C_k)^e_{\lambda}$. Note that this implies $u((C_k)^e_{\lambda}) = W^{e-1}_d$.\\
\\
We see that $(C_k)^e_{\lambda}$ must be irreducible since it is the closure of the set $\{D \in \text{Div}^{\lambda}(C_k) : \text{enc}(D) \leq e \text{ and }u(D) \in W^{e-1}_d\setminus W^e_d\}$ which, by the identification of the Picard sheaf in Theorem \ref{sheaf-id}, is evidently a projective space bundle over a smooth, irreducible (since we assume $\rho(g,R_d,d) \not= 0$) base and thus irreducible.\\
\\
Finally, there can be no pairwise containments among the $(C_k)^e_{\lambda}$'s. To see this, let $e < f \in \mathcal{E}(k,\lambda)$, choose line bundles $L_1 \in W^{e-1}_d\setminus W^e_d$ and $L_2 \in W^{f-1}_d\setminus W^f_d$ and let $D_1 \in |N_{L_1}|$ and $D_2 \in |N_{L_2}|$ such that $\text{enc}(D_2) = f$. We see that $D_1 \not\in (C_k)^f_{\lambda}$ since $u((C_k)^f_{\lambda}) = W^{f-1}_d$ and $L_1 \not\in W^{f-1}_d$. And since $\text{enc}(D_2) > e$, Theorem \ref{deform} implies it cannot be in $(C_k)^e_{\lambda}$.\\
\\
Since for $e \in \mathcal{E}(k,\lambda)$ the subvarieties $(C_k)^e_{\lambda}$ are closed, irreducible, pairwise distinct and have union equal to $\text{Div}^{\lambda}(C_k)$, they form the claimed irreducible decomposition.\\
\\
In the case that $\rho(g,R_d,d) = 0$ the same argument implies that $\text{Div}^{\lambda}(C_k)$ is still the stated union and all the subvarieties in that union are still closed, irreducible and pairwise distinct \textit{except} $(C_k)^{R_d+1}_{\lambda}$ which is no longer irreducible - we deal with it in Corollary \ref{components}.
\end{proof}

The next theorem describes the intersections of the components identified in the previous theorem.

\begin{theorem}[Component intersections]\label{inters}
  Let $C$, $g$,$d$ and $k$ be as in Theorem \ref{irreds} and let $e,f \in \mathcal{E}(k,\lambda)$ with $e < f$. Let $u : \text{\emph{Div}}^{\lambda}(C_k) \rightarrow \text{\emph{Pic}}^{\lambda}(C_k)$ be the Abel-Jacobi map. Then we have:
$$
u\left((C_k)^e_{\lambda}\cap(C_k)^f_{\lambda}\right) = W^{f-1}_d(C)
$$
and for any $L \in W^{f-1}_d(C)$:
$$
u^{-1}(L) \cap (C_k)^e_{\lambda}\cap (C_k)^f_{\lambda}\cong \left\{\begin{array}{cl}\text{\emph{Sub}}_e(\wedge^k \mathbb{C}^{r + 1}) & \text{if $\lambda = n(d)$}\\\\\text{\emph{Sub}}_e(S^k\mathbb{C}^{r + 1}) & \text{if $\lambda = t(d)$}\end{array}\right.
$$
where $r = h^0(L) - 1$.
\end{theorem}
\begin{remark}
Recall that $\text{Sub}_e(\wedge^2\mathbb{C}^{r + 1}) = \text{Sec}_{e/2}G(2,r + 1)$ and $\text{Sub}_e(S^2\mathbb{C}^{r + 1}) = \text{Sec}_{e/2}(\nu_2(\mathbb{P}^r))$ (for $\nu_2$ the quadratic Veronese map) are varieties of $(e/2)$-secant-$(e/2-1)$-planes.
\end{remark}

\begin{proof}[Proof of Theorem \ref{inters}]
  Since $u((C_k)^e_{\lambda}) = W^{e-1}_d$ (as we saw in the proof of Theorem \ref{irreds}, we have that $u((C_k)^e_{\lambda}\cap(C_k)^f_{\lambda}) = W^{e-1}_d\cap W^{f-1}_d = W^{f-1}_d$ since $e < f$ implies $W^{f-1}_d \subset W^{e-1}_d$.\\
  \\
  Since the Abel-Jacobi map is none other than the projection of the projective bundle $\mathbb{P}\wedge^k\mathcal{F}_d$ we know that $u^{-1}(L) = \mathbb{P}(\wedge^k\mathcal{F}_d)_{L} = \mathbb{P}\wedge^kH^0(C,L)^{\vee} \cong |N_L|$. Hence
  $$
  u^{-1}(L) \cap (C_k)^e_{\lambda}\cap (C_k)^f_{\lambda} = \{D \in |N_L| : \text{enc}(D) \leq e,f\}
  $$
which is simply $\text{Sub}_e(\wedge^kH^0(L))$ since $e < f$.
\end{proof}

\begin{corollary}[Component count]\label{components}
Keep the hypotheses of Theorems \ref{irreds} and \ref{inters} and
recall the definitions of $\mathcal{R}_d$, $R_d$ and $r_d$ from
section \ref{bn-petri}. Then, except possibly when $\rho(g,R_d,d) = 0$, 
\begin{itemize}
\item there are $\left\lfloor|\mathcal{R}_d|/2\right\rfloor + \epsilon$ irreducible components of both $\text{\emph{Div}}^{n(d)}(C_2)$ and $\text{\emph{Div}}^{t(d)}(C_2)$ where $\epsilon = 0$ if $r_d$ is odd, and $\epsilon = 1$ if $r_d$ is even.
\item for $k\geq 3$ there are $(|\mathcal{R}_d|-1)-(k-r_d)$ irreducible components of $\text{\emph{Div}}^{n(d)}(C_k)$ and $|\mathcal{R}_d|-1$ irreducible components of $\text{\emph{Div}}^{t(d)}(C_k)$.
\end{itemize}
In the event that $\rho(g,R_d,d) = 0$ (i.e. when $(d-g-1)^2 + 4d$ is a square) we have that $W^{R_d}_d(C)$ is zero-dimensional and over each point is a distinct component of $\text{\emph{Div}}(C_k)$ for all $k\geq 2$. So we must increase the component counts above by $g!\cdot\lambda(g,R_d,d) - 1$ where
$$
\lambda(g,r,d) = \prod_{i=0}^r\frac{i!}{(g-d+r+i)!}
$$
is $(1/g!)$ times the degree of the zero-dimensional scheme $W^{R_d}_d(C)$ (as defined in \cite[pg. 235]{griffhar}). For Petri-general $C$ that scheme is a disjoint union of distinct points.
\end{corollary}

\begin{proof}
This corollary follows from counting the range of values $e$ used in the decomposition described in Theorem \ref{irreds}. The correction then deals with the few instances not covered by part (2) - it follows from the calculation of the class $w^{R_d}_d$ in \cite[pg. 235]{griffhar} which counts the number of points in $W^{R_d}_d$ when it is zero-dimensional.
\end{proof}

\begin{example}
Recall Example \ref{eg}: in that case we described
$\text{Div}^{\lambda}(C_2)$ for $\lambda = c_1(K_{C_2}) = n(2g-2)$. In this
case, $r_{2g-2} = g-2$ and $R_{2g-2} = g-1$ so we do indeed have
$\rho(g,R_{2g-2},2g-2) = 0$. By Corollary \ref{components} we should
therefore expect
$$
\lfloor |\mathcal{R}_{2g-2}|/2\rfloor + \epsilon +
g!\cdot \lambda(g,R_{2g-2},2g-2) - 1
$$
components where $\epsilon = 0$ if
$g$ is odd and $\epsilon = 1$ if $g$ is even. By Equation
\ref{wrd-count} we have $|\mathcal{R}_{2g-2}| = (g-1) - (g-2) + 1 = 2$
and by the equation for $\lambda(g,r,d)$ in the corollary, we have
$\lambda(g,R_{2g-2},2g-2) = 1/g!$. Hence the component count is
\begin{align*}
\lfloor |\mathcal{R}_{2g-2}|/2\rfloor + \epsilon +
g!\cdot \lambda(g,R_{2g-2},2g-2) - 1 & = \lfloor 2/2\rfloor + \epsilon
                                       + g!\cdot 1/g! - 1\\
  & = 1 + \epsilon\\
  & = \left\{\begin{array}{cl}1 & g\text{ odd}\\2 & g\text{ even}\end{array}\right.
\end{align*}
which coincides with the descriptions given in Example \ref{eg} and
with the related results and example of Beauville for surfaces in \cite[\S4]{be}.
\end{example}

Before stating a final conclusion on dimension, we construct a \textit{relativized} version of the
desingularization of our subspace varieties described in section
\ref{lin-alg}. In this setting, the Brill-Noether variety $G^{e-1}_d = G^{e-1}_d(C)$ will play the relative analogue of the Grassmannian $G(e,V)$. $G^{e-1}_d$ comes equipped with a \textit{universal family of $g^{e-1}_d$'s} on $C$ (see \cite[p. 183]{acgh1}). This consists of the data of a line bundle $\mathscr{L}$ on $C\times G^{e-1}_d$ and a rank $e$ \textit{universal sub-bundle} $\mathcal{S} \subset \phi_*\mathscr{L}$, where $\phi$ is the projection $C\times G^{e-1}_d \rightarrow G^{e-1}_d$, such that for each $[V] \in G^{e-1}_d$ the homomorphism
$$
\mathcal{S}\otimes \mathbb{C}([V]) \rightarrow H^0(\phi^{-1}([V]),\mathscr{L}\otimes \mathscr{O}_{\phi^{-1}([V])})
$$
is injective. Let $c : G^{e-1}_d \rightarrow \text{Pic}^d(C)$ be the natural map sending $V \subset H^0(C,L)$ to $L$. By its universal property, $\mathcal{F}_d$ commutes with base-change, hence
$$
\text{\underline{Hom}}(c^*\mathcal{F}_d,\mathcal{N}) \xrightarrow \cong \phi_*(\mathscr{L}\otimes \phi^*\mathcal{N})
$$
for any quasi-coherent $\mathcal{N}$ on $G^{e-1}_d$. Hence, $\text{\underline{Hom}}(c^*\mathcal{F}_d,\mathcal{S}^{\vee}) \cong \phi_*\mathscr{L}\otimes \mathcal{S}^{\vee}\cong \text{\underline{Hom}}(\mathcal{S},\phi_*\mathscr{L})$ so the given inclusion of $\mathcal{S}$ in $\phi_*\mathscr{L}$ yields a quotient $c^*\mathcal{F}_d \rightarrow \mathcal{S}^{\vee} \rightarrow 0$. From this we get a quotient of exterior powers:
$$
c^*\hspace{-2pt}\wedge^k\mathcal{F}_d \rightarrow \wedge^k\mathcal{S}^{\vee} \rightarrow 0
$$
So we get a natural inclusion map $\mathbb{P}(\wedge^k\mathcal{S}^{\vee}) \rightarrow \mathbb{P}(c^*\hspace{-2pt}\wedge^k\mathcal{F}_d)$ which we can compose with $\mathbb{P}(c^*\hspace{-2pt}\wedge^k\mathcal{F}_d) \rightarrow \mathbb{P}\wedge^k\mathcal{F}_d = \text{Div}^{n(d)}(C_k)$ to get:
\begin{equation}\label{surjection}
\phi_e : \mathbb{P}(\wedge^k\mathcal{S}^{\vee}) \rightarrow \text{Div}^{n(d)}(C_k)
\end{equation}
Using Theorems \ref{irreds} and \ref{inters}, this map surjects onto $(C_k)^e_{n(d)} \subset \text{Div}^{n(d)}(C_k)$ - one can see this by, for example, taking $V = H^0(C,L)$ and concluding the surjection fiberwise from the case covered in section \ref{lin-alg}. Again, one can make an analogous construction in the symmetric setting to get $\psi_e : \mathbb{P}(S^k\mathcal{S}^{\vee}) \rightarrow \text{Div}^{t(d)}(C_k)$ surjecting onto $(C_k)^e_{t(d)}$.\\
\\
By the fact that the following diagram (and its analogue in the
symmetric case) is easily seen to commute
\begin{center}
  \begin{tikzcd}
    \mathbb{P}(\wedge^k\mathcal{S}^{\vee}) \arrow[r,twoheadrightarrow,"\phi_e"]\arrow[d]
    & (C_k)^e_{n(d)} \arrow[r,symbol=\subseteq] \arrow[d] & \text{Div}^{n(d)}(C_k)\arrow[d,"u"]\\
    G^{e-1}_d \arrow[r,twoheadrightarrow,"c"] & W^{e-1}_d \arrow[r,symbol=\subseteq] & \text{Pic}^d(C)
  \end{tikzcd}
\end{center}
and the result of Theorem \ref{inters} that the components
$(C_k)^e_{\lambda}$ are fibered in subspace varieties, we see that the
maps $\phi_e$ and $\psi_e$ are \textit{desingularizations} whenever $k
\geq 3$ (and $e \not= k+1$ for $\phi_e$) since they are
desingularizations fiberwise along $u$ (see section \ref{lin-alg}) and the domains of the maps are smooth.\\
\\
Though less would suffice, this in particular makes the dimensions of
the $(C_k)^e_{\lambda}$'s immediately clear:
\begin{theorem}\label{comp-desing}
For $C$ Petri-general, $k \geq 3$ and $e \in \mathcal{E}(k,\lambda)$
the irreducible components $(C_k)^e_{\lambda}$ have dimensions as follows:
    $$
\text{\emph{dim} }(C_k)^e_{\lambda} = \rho(g,e-1,d) + {e + \epsilon\choose k} - 1
$$
where $\epsilon = 0$ if $\lambda = n(d)$ and $\epsilon = k-1$ if $\lambda = t(d)$.
\end{theorem}

\noindent By Remark \ref{kempf-res} we know that for $k \geq 3$ the \textit{fibers} of $u :
(C_k)^e_{\lambda} \rightarrow W^{e-1}_d$ are normal and Cohen-Macaulay,
and - for $k \geq 3$ - have at worst rational singularities.

\section{Examples}

Finally we present some examples to illustrate the results above.\\
\\
Before interpreting the results above in full generality, we return briefly to the projective plane to study some less standard behavior for symmetric cubes:

\begin{example}[Symmetric cube of plane quintic]
The symmetric cube $C_3$ of a smooth plane quintic curve is the first
example of a symmetric product with exorbitant canonical linear series
where the intersection of the main
paracanonical system with the canonical linear series is a \textit{proper}
subspace-variety - i.e. cannot be described as a mere secant variety
of the appropriate Grassmannian, as in the case of symmetric squares.\\
\\
The canonical divisors on $C$ are cut out by conics in the
plane. On the symmetric square of $C$ we can produce canonical
divisors first by taking pencils of these conics in the plane, but
there are others - the pencils form $\mathbb{G}(1,5) \subsetneq
\mathbb{P}\wedge^2H^0(\mathscr{O}_C(2))$. On the symmetric cube of $C$ we can produce canonical divisors first by taking nets of these conics in the plane and, for each conic in such a net, take the collection of triples of points in its intersection with $C$ - in this way we sweep out a canonical divisor in $C_3$. These canonical divisors can all be deformed algebraically since they are cut out by sections which are decomposable in $\wedge^3H^0(K_C) \cong H^0(K_{C_3})$. Relatedly, canonical divisors produced using these nets only form the Grassmannian $\mathbb{G}(2,5)$ in $|K_{C_3}|$ - the remaining canonical divisors are not so easy to describe geometrically. Nevertheless, a larger family of canonical divisors than just those in the Grassmannian are actually deformable. The full family of deformable canonical divisors in $\text{Sub}_5(\wedge^3H^0(K_C)) \subset |K_{C_3}|$ which is a $14$-dimensional singular subvariety which contains the Grassmannian.\\
\\
The fact that the deformable canonical family here is not a secant variety as in the case of symmetric squares is ultimately a result of the fact that $3$-vectors (elements of $\wedge^3V$) can have enclosing dimension which is \textit{not} a multiple of $3$ - those $3$-vectors of fixed such enclosing dimension for a subspace variety which both contains and is contained in secant varieties of the Grassmannian.
\end{example}

\begin{example}[Symmetric products have exorbitant canonical bundle]\label{main-par}
We note in particular that the results above yield a dimension
calculation for the main paracanonical system
$\mathcal{P}_{\text{main}}$ of the symmetric product $C_k$ (that is,
the unique component of $\text{Div}^{\kappa}(C_k)$ dominating
$\text{Pic}^{\kappa}(C_k)$, for $\kappa = c_1(K_{C_k})$). We have $\text{dim } \mathcal{P}_{\text{main}} = \text{dim }\text{Pic}^{\kappa}(C_k) + \text{dim }|N_L|$ for $L$ a generic paracanonical bundle on $C$. This yields $\text{dim } \mathcal{P}_{\text{main}} = g + {g-1\choose k} - 1$. On the other hand, $\text{dim }|K_{C_k}| = {g\choose k}-1$. So we have:
\begin{align*}
\text{dim }|K_{C_k}| -  \text{dim } \mathcal{P}_{\text{main}}& = {g \choose k} - 1 - (g + {g-1\choose k} - 1)\\
& = {g - 1\choose k-1} - g
\end{align*}
which is negative for $k = 2$ but positive in general for $k \geq 3$,
which means it is impossible for the canonical linear series to be
contained in $\mathcal{P}_{\text{main}}$.\\
\\
Importantly, our results not only guarantee this exorbitance, but
indicate precisely the intersection of the main paracanonical system
and the canonical linear series (what Castorena and Pirola
suggestively call the \textit{locus of deformable canonical divisors}
in \cite[Def. 5.3]{castpir}):
$$
\mathcal{P}_{\text{main}} \cap |K_{C_k}| =
\text{Sub}_{g-1}(\wedge^kH^0(C,K_C))
$$
Non-degeneracy of these subspace varieties, easily seen since they
contain the corresponding Grassmannian in its Pl\"{u}cker embedding,
is as expected by \cite[Prop. 5.4]{castpir}, and their dimension
(calculated in Lemma \ref{dimension}) implies that the codimension of
$\mathcal{P}_{\text{main}}\cap|K_{C_k}|$ in $|K_{C_k}|$ is
$$
\text{codim
}\text{Sub}_{g-1}(\wedge^kH^0(K_C)) = {g-1\choose k-1} -
(g-1)
$$
This, in contrast to the case for $k = 2$, is independent of
the parity of $g$ and constitutes a concrete example where the
intersection is not a hypersurface, contrasting with \cite[Thm. 1.3(ii)]{lpp2}.\\
\\
Note also that this example applies even without the Petri-general
hypothesis on $C$ since injectivity of the Petri map $\mu_0$ (see
Definition \ref{petri-map}) always holds for
paracanonical bundles $L \in \text{Pic}^{2g-2}(C)$ for all smooth
projective $C$.
\end{example}

\begin{example}[The Full Picture]
To focus on a relatively concrete but illustrative case, we will let $C$ be a Petri-general curve of genus $g = 37$. We illustrate Theorems \ref{irreds} and \ref{inters} and Corollary \ref{components} by studying systems $|N_L|$ for $L \in \text{Pic}^{g-1}(C)$, which is the translate of the Picard variety in which the theta divisor can be naturally thought to live. Up to translations, $\Theta = W^0_{g-1}$.  We have $\rho(g,d,r) = 37 - (r+1)^2$, hence we have nontrivial Brill-Noether loci $W_{g-1}, W^1_{g-1},\ldots,W^5_{g-1}$ of dimensions $36$, $33$, $28$, $21$, $12$ and $1$ respectively. We will consider $\text{Div}^{n(g-1)}(C_2)$ and $\text{Div}^{n(g-1)}(C_3)$.\\
\\
\textbf{The symmetric square (case $k = 2$)}: we have simply $e_1,e_2,e_3 = 2,4,6$ respectively, and $r_1,r_2,r_3 = 1,3,5$ respectively. $\text{Div}^{n(g-1)}(C_2)$ therefore has three components $(C_2)^2, (C_2)^4$ and $(C_2)^6$ supported over $W^1_{g-1}$, $W^3_{g-1}$ and $W^5_{g-1}$, respectively (via the Abel-Jacobi map $u : \text{Div}^{n(g-1)}(C_2) \rightarrow \text{Pic}^{g-1}(C)$). The component $(C_2)^2$ is:
\begin{itemize}
\item a \textit{top-dimensional} irreducible component of $\text{Div}^{n(g-1)}(C_2)$ (of dimension $33$) which is birational to $W^1_{g-1} \setminus W^2_{g-1}$ (via the Abel-Jacobi map)
\item a $\mathbb{P}^2$-bundle locally over $W^2_{g-1}\setminus W^3_{g-1}$
\item a $G(2,4)$-bundle locally over $W^2_{g-1}\setminus W^3_{g-1}$
\item a $G(2,5)$-bundle locally over $W^3_{g-1}\setminus W^4_{g-1}$
\item (finally) a $G(2,6)$-bundle locally over the whole curve $W^5_{g-1}$.
\end{itemize}
The locus $(C_2)^4$ is a $21 + 5 = 26$-dimensional irreducible component of $\text{Div}^{n(g-1)}(C_2)$ which is:
\begin{itemize}
\item a $\mathbb{P}^5$-bundle locally over $W^3_{g-1}\setminus W^4_{g-1}$
\item a $\mathbb{P}^9$-bundle locally over $W^4_{g-1}\setminus W^5_{g-1}$
\item (finally) a $\text{Sec}_3\hspace{1pt}G(2,6)$-bundle over the whole curve $W^5_{g-1}$.
\end{itemize}
Lastly, $(C_2)^6$ is a $\mathbb{P}^{14}$-bundle over $W^5_{g-1}$.\\
\\
\textbf{The symmetric cube (case $k = 3$)}: we have $e_1,e_2,e_3 = 3,5,6$ respectively, and $r_1,r_2,r_3 = 2,4,5$ respectively. $\text{Div}^{n(g-1)}(C_3)$ therefore has three components, $(C_3)^3$, $(C_3)^5$ and $(C_3)^6$ supported over $W^2_{g-1}$, $W^4_{g-1}$ and $W^5_{g-1}$, respectively. The component $(C_3)^3$ is:
\begin{itemize}
\item a \textit{top-dimensional} irreducible component of $\text{Div}^{n(g-1)}(C_3)$ (of dimension $28$) which is birational to $W^2_{g-1}\setminus W^3_{g-1}$
\item a $\mathbb{P}^3$-bundle over $W^3_{g-1}\setminus W^4_{g-1}$
\item a $G(3,5)$-bundle over $W^4_{g-1} \setminus W^5_{g-1}$
\item (finally) a $G(3,6)$-bundle over the whole curve $W^5_{g-1}$.
\end{itemize}
The component $(C_3)^5$ is an irreducible component of dimension $12 + 9 = 21$ which is:
\begin{itemize}
\item a $\mathbb{P}^9$-bundle locally over $W^4_{g-1}\setminus W^5_{g-1}$
\item (finally) a $\text{Sub}_5(\wedge^3\mathbb{C}^6)$-bundle over the whole curve $W^5_{g-1}$.
\end{itemize}
Lastly, $(C_3)^6$ is a $\mathbb{P}^{19}$-bundle over $W^5_{g-1}$.\\
\\
Note that $\text{Sub}_5(\wedge^3\mathbb{C}^6)$ is a subspace variety of dimension 14 which properly contains the Grassmannian $G(3,6)$ (dimension 9) and is properly contained in the chordal (i.e. 2-secant) variety $\text{Sec}_2\hspace{1pt}G(3,6) = \mathbb{P}^{19}$ (because secant varieties of Grassmannians $G(k,n)$ for $k > 2$ are \textit{not} deficient - see \cite[Theorem 2.1]{cgg}).
\end{example}

\begin{example}[Base locus of $\mathcal{P}_{\text{main}}$]
In \cite[Corollary 1.4]{lpp2}, Lopes-Pardini-Pirola show that on a
general type surface with no irrational pencils of large genus, the
base locus $Z_{\kappa}$ of the main paracanonical system (see Example
\ref{main-par}) is contained in that of the canonical linear
series. In \cite[Proposition 1.3]{castpir}, Castorena-Pirola generalize
to higher dimensions. Given the possibility that $|K_X|$ is
exorbitant, this containment is far from obvious. One might wonder if
this phenomenon can be observed on an appropriately chosen symmetric
product. In fact it cannot: specifically, the main paracanonical
system on $C_k$ will never have a base locus for $k < g$ (which is the
interesting range where exorbitance is a priori possibility). This is because any point $D$ in the base locus would necessarily correspond to a degree $k$ divisor on the curve $C$ failing to impose $k < g$ independent conditions on \textit{all} line bundles $L$ of degree $2g-2$. By Riemann-Roch, this would imply that all divisors of degree $k$ are effective which is not true for any curve $C$ since the effective line bundles of degree $k < g$ form a $k$-dimensional closed subvariety of $\text{Pic}^k(C)$.
\end{example}

\begin{example}[Fano surfaces of cubic threefolds]
The Fano surface $F(X)$ of lines in a smooth cubic threefold $X$ is
smooth with irregularity $5$. By \cite[Corollary 1.5]{lpp2} its canonical series is not
exorbitant (hence its paracanonical system is an irreducible variety of dimension $p_g(F(X)) = 10$ and every canonical divisor is therefore deformable). However, deforming $X$ to a nodal cubic $X_0$ causes $F(X)$ to
deform to a singular surface $F(X_0)$ whose normalization is $\nu: C_2 \rightarrow F(X_0)$, for $C$ a
non-hyperelliptic genus 4 curve - the latter surface \textit{does} have
exorbitant canonical series, as seen in Example \ref{eg}. How do we reconcile these facts? Of course, normalizations
do not necessarily behave well in flat families. What we can say is that if a canonical divisor $D$ on $C_2$ had image $\nu(D)$ in $F(X_0)$ which deformed in a flat family (out of its linear equivalence class), one could lift the family to an algebraic deformation of $D$. Hence the non-deformable canonical divisors on $C_2$ will have non-deformable images on $F(X_0)$ and so $\text{Div}^{\kappa}(F(X_0))$ (indeed, the connected component of the Hilbert Scheme $\text{Hilb}^{\kappa}(F(X_0))$ containing canonical divisors) will be reducible - the canonical series on $F(X_0)$ will be exorbitant. Given the situation for smooth $X$, this demonstrates the possibility of developing exorbitance in families when singularities are introduced.
\end{example}

\begin{example}[Resolving the singular strata of theta divisors]
Recall that the Brill-Noether locus $W^0_{g-1}(C)$ is naturally identified with (a translate of) the theta divisor of $C$ in its Jacobian $J(C) \cong \text{Pic}^{g-1}(C)$. As a result of the Riemann Singularity Theorem (see \cite[pg. 226]{acgh1}) we know that for $C$ Petri-general, the singular locus of $W^0_{g-1}$ is $W^1_{g-1}$ and consequently the Abel-Jacobi mapping $C_{g-1} \rightarrow W^0_{g-1}$ is a resolution of singularities. By what we have said above, a similar phenomenon occurs for the singular locus of $W^1_{g-1}$ when using the divisor variety of the symmetric square of the curve, for the singular locus of $W^2_{g-1}$ when using the symmetric cube, and so on. Specifically, we have:
  \begin{corollary}
    For $\lambda = n(g-1)$, the component $(C_{r+1})^{r+1}_{\lambda}$ of $\text{Div}^{n(g-1)}(C_{r+1})$ is smooth and the Abel-Jacobi mapping $(C_{r+1})^{r+1}_{\lambda} \rightarrow W^r_{g-1}$ is a resolution of singularities.
  \end{corollary}
\end{example}

\begin{example}[Paracanonical system on \'{e}tale and branched double covers of $C_2$]
  Let $\pi: X \rightarrow C_2$ be a smooth double cover, branched along some divisor $B$ whose associated line bundle is a square (though the root, which we will still denote by $B/2$, may not be effective). By Riemann-Hurwitz, we have
  $$
K_X = \pi^*(K_{C_2} + B/2)
$$
so that by the projection formula, we have:
\begin{align*}
  H^0(K_X) & = H^0(K_{C_2}) \oplus H^0(K_{C_2} + B/2)
\end{align*}
(since $\pi_*\mathscr{O}_X \cong \mathscr{O}_{C_2}\oplus \mathscr{O}_{C_2}(-B/2)$).\\
\\
With a little work analogous to what has gone before in this paper, and assuming $B$ is chosen so that $\text{Pic}^0(X) \cong \text{Pic}^0(C_2)$ (which is often the case, but does fail if for example $B = \Delta$), this decomposition (or, more precisely, its dual) globalizes to an analogous statement for Picard sheaves, so we get:
$$
\mathcal{F}_{c_1(K_X)} = \mathcal{F}_{c_1(K_{C_2})}\oplus \mathcal{F}_{\lambda}
$$
where $\lambda := c_1(K_{C_2} + B/2)$. There is some care to be taken here given that Picard sheaves are only well-defined up to twisting by a line bundle, but this will not concern what we conclude here.\\
\\
Given this, we can identify the paracanonical system (let $\kappa := c_1(K_X)$):
$$
\text{Div}^{\kappa}(X) = \mathbb{P}(\mathcal{F}_{c_1(K_{C_2})} \oplus \mathcal{F}_{\lambda})
$$
and when, for example, $B \in |T_L|$ for a (even) degree $d$ line bundle $L$ on $C$, we will have $\mathcal{F}_{c_1(K_{C_2})} = \wedge^2\mathcal{F}_{2g-2}$ and $\mathcal{F}_{\lambda} = \wedge^2\mathcal{F}_{2g-2+d/2}$ for $\mathcal{F}_{2g-2}$ and $\mathcal{F}_{2g-2+d/2}$ Picard sheaves of the curve, for appropriate degrees. Note that $B$ cannot be in $|N_L|$ for any $L$ on $C$ since $N_L$ is never divisible in $\text{Pic}(C_2)$.\\
\\
The deformation result of Theorem \ref{deform} can then be applied in the current setting to conclude that:
\begin{itemize}
\item For $B = T_L$ and $d > 0$ we have:
  $$
|K_X| \cap \text{Div}^{\kappa}(X)_{\text{main}} = \text{Cone}_{\mathbb{P}}(\Sigma)
$$
for $\Sigma =
\text{Sec}_{\lfloor\frac{g-1}{2}\rfloor}\mathbb{G}(1,|K_C|)$,
$\mathbb{P} = \mathbb{P}\wedge^2H^0(K_C + L/2)^{\vee}$ and
$\text{Cone}_{\Lambda}(X)$ denotes the projective cone, with vertex a
projective subspace $\Lambda$, over an embedding of a variety $X$ in projective space. Note here that $L/2$ is well-defined since $B/2$ is.
\item For $B = 0$ (i.e. in the \'{e}tale case) we have:
  $$
|K_X| \cap \text{Div}^{\kappa}(X)_{\text{main}} = \text{Join}(\Sigma,\Sigma)
$$
where $\text{Join}(Y,Z)$ denotes the union of all lines meeting two varieties $Y$ and $Z$ in a fixed projective space, and these two copies of $\Sigma$ lie in the non-intersecting subspaces of $\mathbb{P}H^0(K_X)$ corresponding to the isomorphic summands $\wedge^2H^0(K_C)$ of $H^0(K_X)$.
  \end{itemize}
\end{example}

\bibliography{divisor-varieties}{}
\bibliographystyle{plain}

\end{document}